\documentclass[12pt]{amsart}
\usepackage{amssymb, amstext, amscd, amsmath}
\usepackage[all]{xy}
\usepackage{centernot}
\usepackage{mathtools}
\usepackage{ stmaryrd }

\makeatletter
\def\@cite#1#2{{\m@th\upshape\bfseries%
[{#1\if@tempswa{\m@th\upshape\mdseries, #2}\fi}]}}
\makeatother
\theoremstyle{plain}
\newtheorem{theorem}{Theorem}[section]
\newtheorem{corollary}[theorem]{Corollary}
\newtheorem{proposition}[theorem]{Proposition}
\newtheorem{lemma}[theorem]{Lemma}
\newtheorem{question}{Question}
\newtheorem{problem}{Problem}
\newtheorem{conjecture}{Conjecture}

\theoremstyle{definition}
\newtheorem{definition}[theorem]{Definition}
\newtheorem{example}[theorem]{Example}
\newtheorem{remark}[theorem]{Remark}
\theoremstyle{remark}

\newcommand{\bbA}{{\mathbb{A}}}

\newcommand{\bbC}{{\mathbb{C}}}
\newcommand{\bbD}{{\mathbb{D}}}
\newcommand{\bbI}{{\mathbb{I}}}
\newcommand{\bbF}{{\mathbb{F}}}

\newcommand{\bbN}{{\mathbb{N}}}
\newcommand{\bbQ}{{\mathbb{Q}}}
\newcommand{\bbR}{{\mathbb{R}}}
\newcommand{\bbT}{{\mathbb{T}}}
\newcommand{\bbZ}{{\mathbb{Z}}}

\newcommand{\A}{{\mathcal{A}}}
\newcommand{\B}{{\mathcal{B}}}
\newcommand{\C}{{\mathcal{C}}}

\newcommand{\F}{{\mathcal{F}}}
\newcommand{\G}{{\mathcal{G}}}
\renewcommand{\H}{{\mathcal{H}}}
\newcommand{\I}{{\mathcal{I}}}
\newcommand{\J}{{\mathcal{J}}}
\newcommand{\K}{{\mathcal{K}}}
\renewcommand{\L}{{\mathcal{L}}}
\newcommand{\M}{{\mathcal{M}}}

\renewcommand{\O}{{\mathcal{O}}}
\renewcommand{\P}{{\mathcal{P}}}

\newcommand{\R}{{\mathcal{R}}}
\renewcommand{\S}{{\mathcal{S}}}
\newcommand{\T}{{\mathcal{T}}}
\newcommand{\U}{{\mathcal{U}}}

\newcommand{\X}{{\mathcal{X}}}

\newcommand{\fA}{{\mathfrak{A}}}

\newcommand{\fG}{{\mathfrak{G}}}

\newcommand{\fT}{{\mathfrak{T}}}

\newcommand{\rC}{{\mathrm{C}}}

\renewcommand{\phi}{\varphi}
\newcommand{\upchi}{{\raise.35ex\hbox{\ensuremath{\chi}}}}

\def\ga{\alpha}
\def\gl{\lambda}

\newcommand{\Alg}{\operatorname{Alg}}
\newcommand{\alg}{\operatorname{alg}}
\newcommand{\Aut}{\operatorname{Aut}}

\newcommand{\End}{\operatorname{End}}

\newcommand{\id}{{\operatorname{id}}}

\newcommand{\Lat}{\operatorname{Lat}}

\newcommand{\Rad}{\operatorname{Rad}}

\newcommand{\rep}{\operatorname{rep}}

\newcommand\Span{\mathop{\rm span}}

\newcommand{\ca}{\mathrm{C}^*}
\newcommand{\cenv}{\mathrm{C}^*_{\text{env}}}
\newcommand{\cmax}{\mathrm{C}^*_{\text{max}}}

\newcommand{\Fn}{\mathbb{F}_n^+}

\newcommand{\Gvertex}{{\mathcal{G}}^{0}}

\newcommand{\Gedge}{{\mathcal{G}}^{1}}

\newcommand{\ol}{\overline}

\newcommand{\phist}{u}

\newcommand{\sca}[1]{\left\langle#1\right\rangle}

\newcommand{\nor}[1]{\left\Vert #1\right\Vert}

\newcommand\cpf{\rtimes_{\alpha}\, {\mathcal{G}}}

\newcommand{\vrt}{\G^{(0)}}
\newcommand{\vrtm}{\G^{(0)}_{-}}
\newcommand{\edg}{\G^{(1)}}
\newcommand{\xtau}{X_{\tau}}
\newcommand{\atau}{A_{\tau}}
\newcommand{\ptau}{\phi_{\tau}}

\begin{document}

\title[Non-selfadjoint algebras]{Non-Selfadjoint Operator Algebras: dynamics, classification and $\ca$-envelopes}

\author[E.G. Katsoulis]{Elias~G.~Katsoulis}
\address {Department of Mathematics
\\East Carolina University\\ Greenville, NC 27858\\USA}
\email{katsoulise@ecu.edu}


\maketitle


\section{introduction}
This paper is an expanded version of the lectures I delivered at the Indian Statistical Institute, Bangalore, during the OTOA 2014 conference. My intention at the conference was to offer a gentle introduction to non-selfadjoint operator algebras, using topics that relate to my current research interests. One of my main goals was to provide to the audience most of the prerequisites for understanding the proof of Theorem~\ref{mainthm1}, which identifies the $\ca$-envelope of a tensor algebra as the corresponding Cuntz-Pimsner $\ca$-algebra. (Providing these prerequisites meant of course that I had to survey a good deal of my operator algebra toolkit.) Another goal was to demonstrate (through their classification) that certain non-selfadjoint algebras store a great deal of information about dynamical systems, in a much better way than their $\ca$-counterparts do. These remain two of the main goals of these notes as well. During the preparation of this manuscript however, it occurred to me that there is something else that should be included here. Recently Chris Ramsey and I were able to extend the concept of a crossed product from $\ca$-algebras to arbitrary operator algebras in such a way that many of the selfadjoint results are being preserved in this extra generality, e.g., Takai duality. This opens a new, exciting and very promising area of research that somehow never attracted the attention it deserved. Describing these developments has become yet another goal of these notes.

The paper is divided into eight sections, with this introduction being Section 1. In Section 2 several motivating examples of operator algebras are presented. In Section 3 we see the precise definitions for a $\ca$-correspondence and the various operator algebras associated with it. We also state the various gauge invariance uniqueness theorems of Katsura and others. In Section 4 we describe a very general process for creating injective $\ca$-correspondences from non-injective ones, without straying away from the associated Cuntz-Pimsner algebras (up to Morita equivalence). In Section 5 we give a complete treatment (including proofs) for the $\ca$-envelope of a unital operator space. In Section 6 we describe various classification schemes for operator algebras coming from dynamical systems. In Section 7, we present a new topic that was not covered during the meeting at Bangalore and was mentioned earlier: crossed products of arbitrary operator algebras. The last section of the paper deals with local maps and gives us an opportunity to apply the concepts and tools developed so far to an area of study that goes back to the early work of Barry Johnson.

This paper is written in such a way that it could be used for self study or as seminar notes for an introduction to non-selfadjoint operator algebras. Hence the first few sections unfold at a slower pace and they are less demanding compared to the last sections where I describe my current research.  I have included as many proofs as possible in an article of this kind. Sometimes these are elementary and at other times only sketchy and covering a special case of the theorem. My intention is to give easy access to a variety of techniques without being fussy about completeness. Nevertheless I hope that the non-expert will look further into the details and the subject in general.

It goes without saying that this is not a comprehensive survey article and therefore many important topics and results are not being covered. The reader might also find that some of the topics covered receive an amount of attention which is perhaps unwarranted. This is only due to personal taste. I hope that these notes will serve as the nucleus for a forthcoming book on the topic that will do justice to both the subject and the mathematicians working on it. 

\vspace{.1in} \noindent \textit{Acknowledgment.} I would like to express my gratitude to the organizers of OTOA 2014, and in particular to Jaydeb Sarkar, for the invitation to speak and also for the outstanding scientific environment, the hospitality and the support they provided during my stay at the Indian Statistical Institute, Bangalore. It was a memorable experience that will stay with me for years to come.

\section{Examples} \label{examples}

In this section we examine various examples that motivate the general theory.
\subsection{The semicrossed product $C_0(X) \times_{\sigma} \bbZ^+$.}

This is a very natural and important class of operator algebras. They were introduced by Arveson \cite{Arv1}, Arveson and Josephson \cite{ArvJ} and their study was formalized by Peters \cite{Pet} in 1984. Since then, these algebras have been investigated in one form or another by many authors \cite{DFK, DavKatMem, DKDoc, DavKatAn, DavKatCr, DonH, DKM, KakKatJFA1, McM, Pow1}

Let $X \subseteq \bbC$ be a locally compact Hausdorff space,
$C_0(X)$ be the continuous functions on $X$ and $\sigma: X \rightarrow X$ a proper continuous map. One can give a ``concrete" definition of $C_0(X) \times_{\sigma} \bbZ^+$ as follows

\begin{definition}
Let $(X, \sigma)$ be as above. Given $x \in X$ and $f \in C_0(X)$, we define
\[
\pi_x(f)=
\left(
\begin{matrix}
f(x) & 0 &0 &\dots \\
0    & f(\sigma(x)) & 0 &\dots \\
0&0& f(\sigma^{(2)}(x)) & \dots \\
\vdots & \vdots & \vdots  &\ddots 
\end{matrix}
\right)
\]
and
\[
S_x=
\left(
\begin{matrix}
0 & 0 &0 &\dots \\
1    & 0 & 0 &\dots \\
0&1& 0 & \dots \\
\vdots & \vdots & \vdots  &\ddots  
\end{matrix}
\right).
\]

In other words, for each $x \in X$, $f \in C_0(X)$
\[
\pi_x(f)\xi = (f(x)\xi_0,(f\circ \sigma)(x)\xi_1, (f\circ\sigma^{(2)})(x)\xi_2, \dots)
\]
and $S_x$ is the forward shift
\[
S\xi = (0, \xi_0, \xi_1, \xi_2, \dots).
\]

The semicrossed product $C_0(X) \times_{\sigma} \bbZ^+$ is defined as the norm closed operator algebra acting on $\oplus_{x \in X}\H_x$ and generated by the
operators
\[
\pi(f) \equiv  \oplus_{x \in X} \pi_x(f) \mbox{ and }
S\pi(f),  f \in C_0(X),
\]
where
$$S \equiv \oplus_{x \in X} S_x.$$
\end{definition}

Note the covariance relation
\[
\pi(f)S = S \pi(f\circ \sigma), \quad f \in C_0(X)
\]

A more ``abstract" definition can be given as follows. Let $(\A, \alpha)$ be a C*-dynamical system, i.e., $A$ is C*-algebra and $\alpha: \A\rightarrow \A$ is a non-degenerate $*$-endomorphism (an endomorphism that preserves approximate units).

\begin{definition} \label{firstcov}
An isometric covariant representation $(\pi, V)$ of the $\ca$-dynamical system $(\A, \alpha)$
consists of a $*$-representation $\pi$ of $\A$ on a Hilbert
space $\H$ and an isometry $V \in B(\H )$ so that
\[
\pi(A)V=V\pi(\alpha(A)), \quad \forall A \in \A.
\]
Similar definitions apply for unitary or contractive covariant representations.
\end{definition}

\begin{definition} Let $(\A, \alpha)$ be a $\ca$-dynamical system. The algebra
$\A\times_{\alpha}\bbZ^{+}$ is the universal operator algebra associated with ``all" covariant
representations of $(\A, \alpha)$, i.e., the universal algebra generated
by a copy of $\A$ and an isometry $V$ satisfying the covariant relations.\end{definition}

Note that each covariant representation $(\pi, V)$ provides a contractive representation $\pi\times V$ of $\A\times_{\alpha}\bbZ^{+}$. Similar definition can be given for the universal operator algebra $\A\times_{\alpha}^{un} \bbZ^{+}$ (resp. $\A\times_{\alpha}^{con} \bbZ^{+}$) associated with ``all" unitary (resp. contractive) covariant
representations.

The result below shows that the two definitions we have given so far for $C_0(X) \times_{\sigma} \bbZ^+$ actually describe the same object.

\begin{theorem}[Peters \cite{Pet}] The representation $\oplus_{x \in X} \, \pi_x \times S_x$ of\break $C_0(X) \times_{\sigma} \bbZ^+$ is isometric.\end{theorem}

\begin{proof} This result is an immediate consequence of the gauge-invariance uniqueness Theorem of Katsura (Theorem \ref{thm;giu2}). It pays however to revisit Peter's original argument, with $\sigma$ being assumed a homeomorphism.

For each $x \in X$ and $f \in C_0(X)$ let
\[
\rho_x(f)\xi = \big(\ldots , (f\circ \sigma^{-1})(x)\xi_{-1},  f(x)\xi_0,(f\circ \sigma)(x)\xi_1, \dots\big)
\]
and $U_x$ is the forward shift on $l^2(\bbZ)$. Let $\rho=\oplus_{x \in X} \, \rho_x$ and $U=\oplus_{x \in X} \, U_x$. It suffices to verify that $\rho\times U$ is isometric on $C_0(X) \times_{\sigma} \bbZ^+$ since it is the wot-limit of representations unitarily equivalent to $\oplus_{x \in X} \, \pi_x \times S_x$.

Because of the amenability of $\bbZ$, the representation $\rho\times U$ is faithful (and so isometric) for the $\ca$-algebraic crossed product $C_0(X) \times_{\sigma} \bbZ$. Therefore for any \textit{unitary} covariant representation $(\pi, W)$ of $(X, \sigma)$, the representation $\pi \times W$ is dominated in norm by $\rho\times U$. If $(\pi, V')$ is an arbitrary covariant representation of $(X, \sigma)$ and $V'=W + V$ its Wold decomposition, then the range space $I-P= W^*W=WW^*$ commutes with $\pi$ because of the covariance relations. So we may examine separately the covariant representations $((I-P)\pi, W)$ and $(P\pi, V)$. The previous considerations show that we only need to focus on the representation $(P\pi, V)$, with $V$ a pure isometry, i.e., a direct sum of forward shifts. However in that case one can see that $P\pi \times V$ is unitarily equivalent to the restriction of a regular representation of the $\ca$-algebra $C_0(X) \times_{\sigma} \bbZ$ and the conclusion follows.
\end{proof}

It turns out that the representation theory of $C_0(X) \times_{\sigma} \bbZ^+$ allows more than just integrated forms of isometric covariant representations. Again this was first discovered by Peters in \cite{Pet}. This result too has been extended by now in many different ways and it follows easily now from more general results. (See for instance \cite{MS}.) However, the original argument and the result itself are still a source of inspiration for this author.

\begin{theorem}[Peters \cite{Pet}] The algebras $\A\times_{\alpha}^{con} \bbZ^{+}$, $\A\times_{\alpha} \bbZ^{+}$ (and $\A\times_{\alpha}^{un} \bbZ^{+}$ in the injective case) are isometrically isomorphic.\end{theorem}

\begin{proof}
I will prove only the isomorphism of $\A\times_{\alpha}^{un} \bbZ^{+}$, $\A\times_{\alpha} \bbZ^{+}$.  Clearly the norm of a finite polynomial in $\A\times_{\alpha} \bbZ^{+}$ is dominated by its norm in $\A\times_{\alpha}^{con} \bbZ^{+}$, since any isometric covariant representation of $(\A, \alpha)$ is necessarily contractive.

To prove the converse, it suffices to show that any contractive covariant representation $(\pi, T)$ of $(\A, \alpha)$ is the restriction on a co-invariant subspace of an isometric covariant representation $(\hat{\pi}, V_T)$ of $(\A, \alpha)$. The desired representation is actually given by the formula
\[
\hat{\pi}(a)=
\left(
\begin{matrix}
\pi(a) & 0 &0 &\dots \\
0    & \pi(\alpha(a)) & 0 &\dots \\
0&0& \pi\big(\alpha^{(2)}(a)\big) & \dots \\
\vdots & \vdots & \vdots  &\ddots
\end{matrix}
\right), \quad a \in \A
\]
and
\[
V_T=
\left(
\begin{matrix}
T& 0 &0 &\dots\\
D_T    & 0 & 0 &\dots \\
0&I& 0 & \dots \\
0& 0& I &\dots   \\
\vdots & \vdots & \vdots  &\ddots
\end{matrix}
\right),
\]
where $D_T=(I-T^*T)^{1/2}$. Verifying that the pair $(\hat{\pi}, V_T)$ satisfies the covariance relations means that we need to check equality at each entry to the left and right of the equation $\hat{\pi}(a)V_T= V_T\hat{\pi}\big(\alpha(a)\big)$. This is immediate for all entries except from the  $(2,1)$-entry. There we need to show that
\[
\pi\big( \alpha(a)\big) D_T= D_T\pi\big( \alpha(a)\big) , \, \, \forall a \in \A.
\]
However,
\[
\pi\big( \alpha(a)\big) T^*T=T^*\pi(a)T=T^*T\pi\big( \alpha(a)\big)
\]
and this suffices.
\end{proof}

Dilation results like the one above have received a lot of attention (and continue to do so) as they allow for a  representation theory that is much richer than that of the selfadjoint theory. Nevertheless we will not be focusing on such results in these notes. See instead \cite{DFK} for a recent comprehensive treatment.

\subsection{The classification problem for semicrossed products} One of the fundamental problems in the study of algebras is their classification up to algebraic isomorphisms. The non-slfadjoint operator algebras are no exception to this rule

The problem of classifying the semicrossed products $C_0(X) \times_{\sigma} \bbZ^+$ as algebras was raised by Arveson \cite{Arv1}. Arveson's idea was to use non-selfadjoint crossed products in order to completely recover the dynamics. (This cannot be done using $\ca$-algebraic crossed products.)

Assume that $\sigma_1$ and $\sigma_2$ are topologically conjugate, i.e.,
there exists a homeomorphism $\gamma: X_1 \rightarrow X_2$ so that
\[\gamma \circ\sigma_1 = \sigma_2 \circ \gamma.
\]
Then the semicrossed products
$C_0 (X_1) \times_{\sigma_1} \bbZ^+$ and $C_0 (X_2) \times_{\sigma_2} \bbZ^+$ are isomorphic as algebras. Hence conjugacy of the systems $(X_1, \sigma_1)$ and $(X_2, \sigma_2)$ is a sufficient condition for the isomorphism of $C_0(X_1) \times_{\sigma_1}\bbZ^+$
  and $C_0(X_2) \times_{\sigma_2}\bbZ^+$. Is it necessary? Partial answers were given in the following cases

\begin{itemize}
\item  Both $X_i$ are compact, both $\sigma_i$ have no fixed points, plus some extra conditions. (Arveson and Josephson \cite{ArvJ}, 1969)
\item
Both $X_i$ are compact and $\sigma_i$ have no fixed points.  (Peters \cite{Pet}, 1985)
\item
Both $X_i$ are compact and the set
\[
\{ x \in X_i \mid \sigma_1(x) \neq x, \sigma_1^{(2)}(x)=\sigma_1(x) \} \]
has empty interior. (Hadwin and Hoover \cite{HadH}, 1988)
\item
Both $X_i$ are locally compact and $\sigma_i$ are homeomorphisms. (Power \cite{Pow1}, 1992)
\end{itemize}

The situation was finally resolved in 2008 by Davidson and the author. Our solution combined ideas of the previous authors regarding the character space with a new approach as to how to work with characters. Whereas the previous authors were using the character space in order to characterize various isomorphism-invariant ideals intrinsically, we instead used $2\times 2$-upper triangular representations. The objective of using representations was to understand what happens to the character space under the action induced  by an isomorphism at the algebra level. This approach also paid dividends in \cite{DavKatMem}.

Let us look closely at the character space  $\M_{(X, \sigma)}$ of the semicrossed product $C_0 (X) \times_{\sigma} \bbZ^+$. Any $\theta \in \M_{(X, \sigma)}$ is determined uniquely by its values on $C_0(X)$ and $VC_0(X)$.  Therefore there exist $x \in X$ and $\lambda \in \bbC$, $|\lambda|\leq 1$, so that $\theta\mid _{C_0(X)}$ is just evaluation on $x \in X$ and $\theta (Vg)= \lambda g(x) $, for any $g \in C_0(X)$. In that case we write $\theta = \theta_{x, \lambda}$ and we think of $\lambda$ as the ``value" of $\theta$ on $V$.\footnote{This statement is precise in the case where $X$ is compact as $V$ belongs to  $C_0 (X) \times_{\sigma} \bbZ^+$.} Note that in the case where $x$ is not a fixed point for $\sigma$ we have that $\lambda = 0$ and so there is only one choice for $\theta_{x, \lambda}$, i.e., $\theta_{x ,0}$. Indeed this follows by  applying $\theta$ to the covariance relation
\[
ge_iV  =  e_iV (g \circ \sigma),
\]
where $\{e_i\}_i$ is an approximate unit for $C_0(X)$ and $g\in C_0(X)$ satisfies $g(x)=1$, $g(\sigma(x))=0$.
If $x \in X$ is not a fixed point for $\sigma$, then there are other possibilities for $\theta_{x, \lambda}$ beyond $\theta_{x, 0}$. The collection of all such characters $\theta_{x, \lambda}$ for a fixed $x \in X$ is denoted as $\M_{(X, \sigma, x)}$.

Now if $\A$ is an algebra, then $\rep_{\fT_2} \A$ will denote the collection of all representations of $\A$ onto $\fT_2$,
the upper triangular $2 \times 2$ matrices.
To each $\pi \in \rep_{\fT_2} \A$ we associate two characters
$\theta_{\pi, 1}$ and $\theta_{\pi, 2}$ which correspond to compressions on the $(1,1)$ and
$(2,2)$-entries, i.e.,
\[
\theta_{\pi, i}(a)\equiv \langle \pi(a)\xi_i ,  \xi_i\rangle, \quad a \in \A, i=1,2,
\]
where $\{\xi_1,\xi_2\}$ is the canonical basis of $\bbC^2$. If $\gamma:\A_1 \rightarrow \A_2$
is an isomorphism between algebras, then $\gamma$ induces isomorphisms,
\begin{alignat}{2}
\gamma_c&:\M_{\A_1} \to \M_{\A_2}
&\quad\text{by}\quad& \gamma_c(\theta) =  \theta\circ\gamma^{-1} \label{induced1}\\
\gamma_r&:\rep_{\fT_2}  \A_1 \to \rep_{\fT_2}  \A_2
&\quad\text{by}\quad& \gamma_r(\pi) =  \pi\circ\gamma^{-1} \label{induced2},
\end{alignat}
which are compatible in the sense that,
\begin{equation} \label{compat1}
\gamma_c (\theta_{\pi, i}) = \theta_{\gamma_r(\pi),i}, \quad i=1,2,
\end{equation}
for any $\pi \in \rep_{\fT_2}  \A_1 $.

Let $\X$ be
a locally compact Hausdorff space and $\sigma$ a proper continuous map. For
$x_1 , x_2 \in \X$, let
\[
\rep_{x_1,x_2}C_0(X)\times_{\sigma} \bbZ^+ \equiv \{\pi \in \rep_{\fT_2} C_0(X)\times_{\sigma} \bbZ^+  \mid \theta_{\pi, i}
\in \M_{(X, \sigma, x_i)}, i=1,2\} .
\]
Clearly, any element of $ \rep_{\fT_2} C_0(X)\times_{\sigma} \bbZ^+$ belongs to $\rep_{x, y}C_0(X)\times_{\sigma} \bbZ^+$ for some
$x,y \in \X$.

\begin{lemma} \label{fundrel}
Let  $\X$ be
a locally compact Hausdorff space and$\sigma$ a proper continuous map on $\X$. Assume that $x,y \in \X$
are not fixed points for $\eta$ and let $\pi \in \rep_{x,y}C_0(X)\times_{\sigma} \bbZ^+$. Then, $y=\sigma(x)$.
\end{lemma}

\begin{proof} By assumption, $\theta_{\pi, 1}= \theta_{x,0}$ and $\theta_{\pi, 2}= \theta_{y,0}$
and so $\theta_{\pi, 1}(gU) = \theta_{\pi, 2}(gU)=0$, for any $g \in C_0(\X)$.
Therefore for each $g \in C_0(\X)$ there exists $c_g \in \bbC$
so that
\[
\pi(gU)=
\left(\begin{array}{cc} 0 & c_g \\  0 & 0 \end{array} \right).
\]
Clearly there exists at least one $g \in C_0(\X)$ so that
$c_g\neq0$, or otherwise the range of $\pi$ would be commutative.
Applying $\pi$ to $gUf= (f\circ\eta) gU$
for $f \in C_0(\X)$ and this particular $g$, we get
\begin{equation*}
\left(\begin{array}{cc} 0 & c_g \\ 0 & 0 \end{array} \right)
\left(\begin{array}{cc} f(x) & t \\ 0 & f(y)\end{array} \right)
=
\left(\begin{array}{cc} f(\eta(x)) & t' \\ 0 & f(\eta(y)) \end{array} \right)
\left(\begin{array}{cc} 0 & c_g \\  0 & 0 \end{array} \right)
\end{equation*}
for some $t, t' \in \bbC$, depending on $f$.
By comparing $(1,2)$-entries, we obtain,
\[
f(y) = f (\eta(x)) , \forall  \in C_0(\X),
\]
i.e., $y = \eta(x)$, as desired.
\end{proof}

\begin{theorem}[Davidson and Katsoulis \cite{DavKatCr}, 2008] \label{DavKat1} Let $X_i$ be a lo-\break cally compact Hausdorff space and let
  $\sigma_i$ a proper continuous map on $X_i$, for $i=1,2$. Then the dynamical
  systems $(X_1, \sigma_1)$ and $(X_2, \sigma_2)$ are conjugate
  if and only if the semicrossed products $C_0(X_1) \times_{\sigma_1}\bbZ^+$
  and $C_0(X_2) \times_{\sigma_2}\bbZ^+$ are isomorphic as algebras.
\end{theorem}

\begin{proof} I will prove the result under the assumption that both $\sigma_i$ have no fixed points. In that case, $\M_{(X_i , \sigma_i , x)}=\{ \theta_{x, 0}\}$, $x \in X_i$, $i=1,2$, and so the character space $\M_{X_i, \sigma_i}$ equipped with the $w^*$-topology is homeomorphic to $X_i$, $i=1,2$.

Assume that there exists an isomorphism
\[
\gamma: C_0(X_1) \times_{\sigma_1}\bbZ^+ \longrightarrow C_0(X_2) \times_{\sigma_2}\bbZ^+.
\]
Then the map
\[
\gamma_c:\M_{X_1, \sigma_1} \to \M_{X_2, \sigma_2}
\]
provides a homeomorphism between the spaces $X_1$ and $X_2$.Furthermore
\[
\gamma_r \big(  \rep_{x,\sigma_1(x)}C_0(X)\times_{\sigma} \bbZ^+\big) =\rep_{\gamma_c(x),\gamma_c(\sigma_1(x))}C_0(X)\times_{\sigma} \bbZ^+.
\]
The previous lemma applied to the right side of the above equation shows that
\[
\sigma_2\big(\gamma_c(x)\big)=\gamma_c(\sigma_1(x))
\]
and the conclusion follows.
\end{proof}

\subsection{The tensor algebra of a graph $\G$} This is a class of algebras that was introduced by Muhly and Solel \cite{MS} under the name quiver algebras. They generalize Poescu's non-commutative disc algebras \cite{Pop, Pop2}. (See below for a precise definition.)

Let $\G=  (\G^{0}, \G^{1}, r , s)$ be a countable directed graph and let $\G_{\infty}$ denote the (finite) path space of $\G$, i.e.,  all vertices $p \in \G^0$ and all finite paths
          $$v = e_ke_{k-1}\dots e_1$$
  where the $e_i \in \G^1$  are edges satisfying $s(e_i)=r(e_{i-1})$, $i=1,2, \dots , k$, $k \in \bbN$.

       Let $\{\xi_v\}_{v \in \G_{\infty}}$ denote the usual orthonormal basis of the Fock space $\H_{\G} \equiv l^2( \G_{\infty} )$, where $\xi_v$ is the characteristic function of $\{v\}$. The left creation operator $L_v$, $v \in \G_{\infty}$, is defined by
\[
L_v \xi_w =
 \left\{ \begin{array}{cl}
\xi_{qp} &\mbox{ if $s(v) = r(w)$} \\
0 &\mbox{ otherwise}.
\end{array} \right.
\]

\begin{definition}The norm closed algebra generated by $\{ L_v \mid v \in \G_{\infty}\}$, denoted as $\T_{\G}^+$, is the tensor tensor algebra of the graph $\G$. Its weak closure, denoted as $\L_{\G}$, is the free semigroupoid algebra of $\G$.
\end{definition}

There is a particular class of graphs that deserves a special mention here. If $\G$ is the graph consisting of one vertex $p$ and a loop $e$, then $L_p=I$ and the left creation operator operator $L_e$ is unitarily equivalent to the forward shift on $l^2(\bbN)$. For this graph $\G$, the tensor algebra $\T^+_{\G}$ is unitarily equivalent to the norm closed algebra generated by the shift operator, i.e., the disc algebra $A(\bbD)$. More generally, if $\G$ is the graph consisting of one vertex $p$ and $n\geq 2$ loop edges $e_1, e_2, \dots , e_n$ then the corresponding creation operators $L_{e_1}, L_{e_2}, \dots, L_{e_n}$ are pure isometries with orthogonal ranges whose joint wandering space is one dimensional (Cuntz-Toeplitz isometries). In this case, the tensor algebra $\T^+_{\G}$ is called the non-commutative disc algebra and it is denoted as $\A_n$. These algebras were introduced by Popescu \cite{Pop} and play an important role in the theory.  As we shall see shortly, $\A_n$ is the universal algebra generated by a row isometry with $n$ entries.

The algebras $\T_{\G}^+$ were classified by Kribs and the author in 2004.

\begin{theorem}[Katsoulis and Kribs \cite{KaKr}, 2004] Let $\G_1$, $\G_2$ be directed graphs with no sinks. Then the tensor algebras $\T_{\G_1}^+$ and $\T_{\G_2}^+$ are isomorphic as algebras if and only if $\G_1$ are $\G_2$ are isomorphic as graphs.
\end{theorem}

Once again a key ingredient of the proof is the use of multiplicative forms and $2 \times 2$-upper triangular representations. Indeed, in a graph algebra $\T^+_{\G}$ one identifies the vertices of the associated graph $\G$ with the connected components of the character space of $\T^+_{\G}$. The dimension of these connected components determines the number of loop edges supported on each one of the corresponding vertices. By looking at $2 \times 2$-upper triangular representations and the multiplicative forms appearing in the diagonal entries, one can decide whether or not there exist arrows between the vertices corresponding to these forms. Calculating the number of these edges though requires an argument. See \cite{KaKr} for more details. Note the arguments in \cite{KaKr} actually inspired the work in \cite{DavKatCr}.

Let see how the above result can be used beyond operator algebra theory. We will reformulate a famous problem in graph theory as an operator algebra problem.

\begin{definition} Let $\G$ be a finite undirected graph with no loop edges or multiple edges between any two of its vertices.  A vertex-deleted subgraph of $\G$ is a subgraph formed by deleting exactly one vertex from $\G$ and its incidence edges.\end{definition}

\begin{definition}  For a graph $\G$, the {deck} of $\G$, denoted as $D(\G)$, is the multiset of all vertex-deleted subgraphs of $\G$. Each graph in $D(\G)$ is called a \textit{card}. Two graphs that have the same deck are said to be hypomorphic or \textit{reconstructions} of each other.\end{definition}
With these definitions at hand, we can state the following well-known conjecture.

\begin{conjecture}[Kelly and Ulam] Any two hypomorphic graphs on at least three vertices have to be isomorphic.
\end{conjecture}

A finite directed graph $\G$ will belong to the subclass $\fG_0$ of all directed graphs if $\G$ comes from a finite undirected graph by replacing each edge with two directed edges with opposite directions. The concepts of a card, a deck and hypomorphism transfer to graphs in $\fG$
and the Reconstruction Conjecture can be stated as

\begin{conjecture}[Reconstruction Conjecture of Kelly and Ulam] Any two hypomorphic graphs in $\fG_0$ on at least three vertices are necessarily isomorphic.
\end{conjecture}

\begin{definition} If $\G \in \fG$, then a vertex-deleted subalgebra of $\T_{\G}^+$ is formed by deleting from $\G$ exactly one vertex and its incidence edges and then taking the subalgebra of $\T_{\G}^+$ formed by the partial isometries and projections corresponding to the  remaining edges and vertices respectively.\end{definition}

\begin{definition} For a tensor algebra $\T_{\G}^+$, the deck of $\T_{\G}^+$, denoted as $D(\T_{\G}^+)$, is the multiset of all vertex-deleted subalgebras of $\T_{\G}^+$. Each graph in $D(\T_{\G}^+)$ is called a \textit{card}. Two tensor algebras that have the same deck are said to be hypomorphic or reconstructions of each other.\end{definition}

In an ongoing collaboration with Gunther Cornelissen at Utreht we have the following

\begin{theorem}If $\G_1, \G_2 \in \fG_0$, then the graphs $\G_1$ and $\G_2$ are hypomorphic if and only $\T_{\G_1}^+$ and $\T_{\G_2}^+$ are hypomorphic as operator algebras.\end{theorem}

Therefore the reconstruction conjecture admits the following equivalent form

\begin{corollary}The reconstruction conjecture in graph theory is\break equivalent to the assertion that hypomorphic tensor algebras of graphs in $\fG_0$ are necessarily isomorphic as algebras. \end{corollary}

There is a more abstract approach to defining the tensor algebra of a graph. The situation is similar to that of semicrossed products and there is a reason for that as we shall see soon.

\begin{definition} \label{defn;CKr}
Let $\G=  (\G^{0}, \G^{1}, r , s)$ ba a countable directed graph.
A family of partial isometries $\{ L_e\}_{ e\in
\G^{(1)}}$ and projections $\{ L_p \}_{p\in \G^{(0)}}$ is said to obey
the Cuntz-Krieger-Toeplitz relations associated with $\G$ if and only if
they satisfy
\[
(\dagger)  \left\{
\begin{array}{lll}
(1)  & L_p L_q = 0 & \mbox{$\forall\, p,q \in \G^{(0)}$, $ p \neq q$}  \\
(2) & L_{e}^{*}L_f = 0 & \mbox{$\forall\, e, f \in \G^{(1)}$, $e \neq f $}  \\
(3) & L_{e}^{*}L_e = L_{s(e)} & \mbox{$\forall\, e \in \G^{(1)}$}      \\
(4)  & L_e L_{e}^{*} \leq L_{r(e)} & \mbox{$\forall\, e \in \G^{(1)}$} \\
(5)  & \sum_{r(e)=p}\, L_e L_{e}^{*} \leq L_{p} & \mbox{$\forall\, p
\in \G^{(0)}$ }
\end{array}
\right.
\]
\end{definition}

\begin{definition}. The tensor algebra $\T_{\G}^+$ of a graph $\G$  is  the universal operator algebra for all families of partial isometries $\{ L_e\}_{ e\in
\G^{(1)}}$ and projections $\{ L_p \}_{p\in \G^{(0)}}$ which obey
the Cuntz-Krieger-Toeplitz relations associated with $\G$.
\end{definition}

It seems that we have given two different definitions for the same object. The next result shows that there are no contradictions arising from this.

\begin{theorem}[Fowler, Muhly and Reaburn \cite{FMR}, 2001] The representation of $\T_{\G}^+$ on the Fock space $\H_G$ is isometric.\end{theorem}

In particular, $\A_n$ is the universal algebra for a row isometry. The above result follows easily as a corollary of Theorem~\ref{thm;giu2}.


\section{$\ca$-correspondences}

The algebras appearing in the previous sections are all examples of tensor algebras of $\ca$-correspondences. In this section we describe in detail that broad class of operator algebras.

The construction of $\ca$-algebras coming from a $\ca$-correspondence originated in the seminal paper of Pimsner \cite{Pimsn}. Pimsner considered only injective correspondences in \cite{Pimsn}. Even though such correspondences are easier to work with, many natural $\ca$-algebras come from non-injective ones. Many mathematicians tried to extend Pimsner's ideas to non-injective correspondences with varied levels of success. Now days Katsura's approach \cite{KatsuraJFA} is considered to be the most successful one; this is the one presented here. We start with the basic definitions.

 Let $A$ be a $\ca$-algebra. An \emph{inner-product right $A$-module} is a linear space $X$ which is a right $A$-module together with a map
\begin{align*}
(\cdot,\cdot)\mapsto
\sca{\cdot,\cdot}\colon X \times X \rightarrow A
\end{align*}
such that
\begin{align*}
\sca{\xi, \gl \zeta + \eta}&= \gl\sca{\xi,\zeta} + \sca{\xi,\eta}\qquad
\\
\sca{\xi, \eta a}&= \sca{\xi,\eta}a \qquad \\
\sca{\eta,\xi}&= \sca{\xi,\eta}^* \qquad \\
\sca{\xi,\xi} &\geq 0 \text{; if }\sca{\xi,\xi}= 0 \text{ then }
\xi=0.
\end{align*}

For $\xi \in X$ we write $\nor{\xi}_X^2 := \nor{\sca{\xi,\xi}}_A$ and one can deduce that $\nor{\cdot}_X$ is actually a norm. If $X$ equipped with that norm is a complete normed space then it will be called \emph{ Hilbert $A$-module}. For a Hilbert $A$-module $X$ we define the set $L(X)$ of the \emph{adjointable maps} that consists of all maps $s:X \rightarrow X$ for which there is a map $s^*: X \rightarrow X$ such that
\[
\sca{s\xi,\eta}= \sca{\xi,s^*\eta}, \,\, \xi, \eta \in X.
\]

The compact operators $K(X) \subseteq L(X)$ is the closed subalgebra of $L(X)$ generated by the "rank one" operators
\[
\theta_{\xi , \eta}(z) := \xi  \sca{\eta ,z}, \quad \xi, \eta ,z \in X
\]

\begin{definition}. A $\ca$-correspondence $(X, A, \phi)$ consists of a Hilbert A-module $(X, A)$ and a non-degenerate left action
\[
\phi: A \longrightarrow L(X).
\]
If $\phi$ is injective then the $\ca$-correspondence $(X, A, \phi)$ is said to be injective.\end{definition}

A (Toeplitz) representation $(\pi,t)$ of $X$ into a $\ca$-algebra $B$, is a pair of a $*$-homomorphism $\pi\colon A \rightarrow B$ and a linear map $t\colon X \rightarrow B$, such that
\begin{enumerate}
 \item $\pi(a)t(\xi)=t(\phi_X(a)(\xi))$,
 \item $t(\xi)^*t(\eta)=\pi(\sca{\xi,\eta}_X)$,
\end{enumerate}
for $a\in A$ and $\xi,\eta\in X$. An easy application of the $\ca$-identity shows that
\begin{itemize}
\item[(iii)] $t(\xi)\pi(a)=t(\xi a)$
\end{itemize}
 is also valid. A representation $(\pi , t)$ is said to be \textit{injective} iff $\pi$ is injective; in that case $t$ is an isometry.


\begin{definition} The \emph{Toeplitz-Cuntz-Pimsner} $\ca$-algebra $\T_{X}$ of a $\ca$-correspondence $(X, A, \phi)$ is the $\ca$-algebra generated by all elements of the form $\pi_{\infty}(a), t_{\infty}(\xi)$, $a \in A$, $\xi \in \X$, where $(\pi_{\infty}, t_{\infty})$ denotes the universal Toeplitz representation of $(X, A, \phi)$.
\end{definition}

The ideas of Pimsner \cite{Pimsn} were brought in the non-selfadjoint world by the pioneering work of Muhly and Solel \cite{MS}, who recognized an important subalgebra of the Toeplitz-Cuntz-Pimsner $\ca$-algebra $\T_{X}$. 

\begin{definition} The \emph{tensor algebra} $\T_{X}^+$ is the norm-closed subalgebra of $\T_X$ generated by all elements of the form $\pi_{\infty}(a), t_{\infty}(\xi)$, $a \in A$, $\xi \in \X$, where $(\pi_{\infty}, t_{\infty})$ denotes the universal Toeplitz representation of $(X, A, \phi)$.
\end{definition}

Let us see now how the examples of Section 2 manifest as tensor algebras of specific $\ca$-correspondences.

\begin{example} (i) Let $\A$ be a unital $\ca$-algebra and $\alpha$ a unital endomorphism. Set $A = \A$, $X_{\alpha}= \A$,
\[\sca{\xi , \eta} = \xi^*\eta, \quad \xi , \eta \in X_{\alpha}
\]
and $\phi(a) \xi  b = \alpha(a)\xi b$, for $a, b \in A$, $ \xi \in X_{\alpha}$. A non-degenerate\footnote{If $(\pi,t)$ is degenerate then restrict on the reducing subspace $\pi(1)$.}
 (Toeplitz) representation $(\pi,t)$ of $(X_{\alpha}, A)$ into a $\ca$-algebra $B$ should satisfy
\[
1 = \pi(1) = \pi(\sca{1,1})= t(1)^*t(1),
\]
where $1\in \A$ is the unit element. Furthermore Properties (i), (ii) and (iii) imply
\[
\begin{split}
\pi(a)t(1)&= \pi(\phi(a)1)=\pi(\alpha(a)) \\
              &=\pi(1\alpha(a))= t(1)\pi(\alpha(a))
              \end{split}
              \]
This implies that the pair $(\pi , t(1))$ is a covariant representation of the $\ca$-dynamical system $(\A, \alpha)$, in the sense of Definition \ref{firstcov}.

Conversely if $(\pi , V)$ is a covariant representation of the $\ca$-dynamical system $(\A, \alpha)$, then set $t(\xi) = V\pi(\xi)$, $\xi \in \A$ and verify that the pair $(\pi,t)$ is a Toeplitz representation  of $(X_{\alpha}, A)$. In this case the tensor algebra of $(X_{\alpha}, A)$ is the semicrossed product $\A \times_{\alpha} \bbZ^+$.

\vspace{.05in}
(ii) Let $\G = (\G^0, \G^1, r, s)$ be a finite graph. Let $A= c_0( \G^0)$, $ X_{\G} = c_0(\G^1)$
\[
\sca{\xi , \eta}(p)=\sum_{s(e)=p} \, \overline{\xi(e)}\eta (e), \,\, \xi, \eta \in X_{\G},  \, p\in \G^0
\]
and $\big( \phi( f) \xi g \big)(e) = f(r(e))\xi(e)h(s(e))$, with $ \xi \in X_{\G}$, $f,g \in A$ and $ e \in \G^1$.

Let  Let $(\pi, t)$ be a non degenerate Toeplitz representation of the \textit{graph correspondence} $(X_{\G} , A)$. Let $L_p=\pi(1_p)$, $p \in \G^0$ and $L_e = t(1_e)$, $e \in \G^1$, where $1_e \in X_{\G}$ denotes the characteristic function of the singleton $\{e\}$, $e \in \G^1$ and similarly for $1_g \in A$, with $g \in \G^0$. Then $L_pL_q=\pi(1_p1_q)=\delta_{p,q}L_p$, where $p,q \in \G^0$.
Also,
\[
L_e^*L_f=t(1_e)^*t(1_f)= \pi(\sca{1_e, 1_f})=0, \,\, \mbox{for }e, f \in \G^1, e\neq f
\]
and similarly $L_e^*L_e=L_{s(e)}$, $e \in \G^1$. It is clear that the family $\{ L_e\}_{ e\in
\G^{(1)}}$  of partial isometries  and  $\{ L_p \}_{p\in \G^{(0)}}$ of projections obey
the Cuntz-Krieger-Toeplitz relations of Definition \ref{defn;CKr}.

Conversely given  a family $\{ L_e\}_{ e\in
\G^{(1)}}$  of partial isometries  and\break  $\{ L_p \}_{p\in \G^{(0)}}$ of projections obeying
the Cuntz-Krieger-Toeplitz relations of Definition \ref{defn;CKr}, we can define a Toeplitz representation of the graph correspondence $(X_{\G}, A)$ by setting $\pi(1_p)=L_p$, $ p \in \G^0$, $t(1_e)= L_e$, $e \in \G^1$ and extending by linearity. In this case the tensor algebra of $(X_{\G}, A)$ is therefore $\T_{\G}^+$.

\vspace{.05in}

(iii) A broad class of $\ca$-correspondences arises naturally from the concept of a topological graph. A topological graph $\G= (\G^{0}, \G^{1}, r , s)$ consists of two $\sigma$-locally compact\footnote{Due to this assumption, all discrete graphs appearing in this paper are countable.} spaces $\Gvertex$, $\Gedge$, a continuous proper map $r: \Gedge \rightarrow \Gvertex$ and a local homeomorphism $s: \Gedge \rightarrow \Gvertex$. The set $\Gvertex$ is called the base (vertex) space and $\Gedge$ the edge space. When $\Gvertex$ and $\Gedge$ are both equipped with the discrete topology, we have a discrete countable graph (see below).

With a topological graph $\G=  (\G^{0}, \G^{1}, r , s)$ there is a $\ca$-correspon-\break dence $X_{\G}$ over $C_0 (\Gvertex)$. The right and left actions of $C_0(\Gvertex )$ on $C_c ( \Gedge)$ are given by
\[
(fFg)(e)= f(r(e))F(e)g(s(e))
\]
for $F\in C_c (\Gedge)$, $f, g \in C_0 (\Gvertex)$ and $e \in \Gedge$. The inner product is defined for $F, G \in C_c ( \Gedge)$ by
\[
\left< F \, | \, G\right>(v)= \sum_{e \in s^{-1} (v)} \overline{F(e)}G(e)
\]
for $v \in \Gvertex$. Finally, $X_{\G}$ denotes the completion of $C_c ( \Gedge)$ with respect to the norm
\begin{equation} \label{norm}
\|F\| = \sup_{v \in \Gvertex} \left< F \, | \, F\right>(v) ^{1/2}.
\end{equation}

When $\Gvertex$ and $\Gedge$ are both equipped with the discrete topology, then the tensor algebra $\T_{\G}^+ \equiv \T^+_{X_{\G}}$ associated with $\G$ coincides with the quiver algebra of Muhly and Solel \cite{MS}. In that case, $\T_{\G}^+$ has already been described.
\end{example}

\begin{theorem} Let $(\pi, t)$ be a representation of a $\ca$-correspondence $(X, A, \phi)$. Then there exists a map
\[
\psi_t \colon K(X)\longrightarrow \ca(\pi , t)
\]
so that $\psi_t (\theta_{\xi, \eta})= t(\xi)t(\eta)^*$, for all $\xi, \eta \in X$. If $\pi$ is injective then $\psi_t$ is injective as well.
\end{theorem}

\begin{proof}
Represent $B\equiv \ca(\pi , t)$ faithfully on a Hilbert space $\H$, and consider the representation of $K(X)$ on the Hilbert space $X\otimes_A\H$. (See below for a precise definition of that tensor product.) It is easy to see now that we have an isometry $\xi\otimes h \mapsto t(\xi)h \in \H$. Then $\psi_t$ is the equivalent representation of $K(X)$ on the range of that isometry.
\end{proof}

\begin{definition} A representation $(\pi, t)$ of a $\ca$-correspondence $(X, A, \phi)$ is said to be a covariant representation iff
\[
\pi(a) =\psi_t(\phi(a)), \,\, \mbox{ for all } a \in J_X,
\]
where $J_X= \phi^{-1}(K(X))\cap (\ker \phi )^{\perp}$.
\end{definition}

\begin{definition}. The \emph{Cuntz-Pimsner} $\ca$-algebra $\O_{X}$ of a $\ca$-cor-break respondence $(X, A, \phi)$ is the $\ca$-algebra generated by all elements of the form $\overline{\pi}_{\infty}(a), \overline{t}_{\infty}(\xi)$, $a \in A$, $\xi \in X$, where $(\overline{\pi}_{\infty}, \overline{t}_{\infty})$ denotes the universal covariant representation of $(X, A, \phi)$.
\end{definition}

It is not clear that a $\ca$ correspondence admits non-trivial Toeplitz or covariant representations. For that purpose we introduce the interior tensor product of $\ca$-correspondences.

The \emph{interior} or \emph{stabilized tensor product}, denoted by $X \otimes X$ or simply by $X^{\otimes 2}$, is the quotient of the vector space tensor product $X \otimes_{\alg} X$ by the subspace generated by the elements of the form
\begin{align*}
\xi a \otimes \eta - \xi \otimes \phi(a)\eta, \quad \xi, \eta \in X, a
\in A.
\end{align*}
It becomes a pre-Hilbert $B$-module when equipped with
\begin{align*}
 (\xi \otimes \eta)a:&= \xi \otimes (\eta a),\\
\sca{\xi_1\otimes \eta_1, \xi_2\otimes \eta_2}:&= \sca{y_1,
\phi(\sca{\xi_1,\xi_2})\eta_2}
\end{align*}

For $s\in L(X)$ we define $s\otimes \id_X \in L(X\otimes X)$ as the mapping $$\xi\otimes y \mapsto s(\xi)\otimes y.$$
Hence $X \otimes Y$ becomes a $\ca$-correspondence by defining $\phi_{X\otimes X}(a):= \phi_X(a) \otimes \id_X$.

The \emph{Fock space} $\F_{X}$ over the correspondence $X$ is the interior direct sum of the $X^{\otimes n}$ with the structure of a direct sum of $\ca$-correspon-\break dences over $A$,
\[
\F_{X}= A \oplus X \oplus X^{\otimes 2} \oplus \dots .
\]
Given $\xi \in X$, the (left) creation operator $t_{\infty}(\xi) \in \L(\F_{X})$ is defined as
\[
t_{\infty}(\xi)( a , \zeta_{1}, \zeta_{2}, \dots ) = (0, \xi a, \xi
\otimes \zeta_1, \xi \otimes \zeta_2, \dots).
\]
 For any $a \in A$, we define
 $$\pi_{\infty}(a) = L_a \oplus \phi(a) \oplus (\oplus_{n=1}^{\infty} \phi(a) \otimes \id_{n}).$$

 It is easy to verify that $( \pi_{\infty}, t_{\infty})$ is a Toeplitz representation of $(X, A)$ which is called the \emph{Fock representation} of $(X, A)$. Note that $\pi_{\infty}$ is faithful and so non-trivial.

We also need to produce non-trivial covariant representations. Our presentation follows that of Katsura \cite{KatsuraJFA}. We require the following

\begin{lemma} \label{inner}
Let $(X, A)$ be a $\ca$-correspondence and $J\subseteq A$ a closed ideal. If $k \in K(X)$, then the following are equivalent
\begin{itemize}
\item[(i)] $k \in K(XJ) \equiv \big[\{ \theta_{\xi a , \eta} \mid \xi, \eta \in X , a \in J\}\big]$
\item[(ii)] $\sca{ k\xi , \eta } \in J , \mbox{ for all } \xi , \eta \in X$
\end{itemize}
\end{lemma}

\begin{proof}
Since the operators satisfying (ii) form an ideal, it is enough to consider $k\geq 0$ satisfying (ii) and verify (i). It is enough show that $k^3 \in K(XJ)$. Indeed, if $k = \lim_n \sum_i \theta_{\xi_{i}^{n} , \eta_i^n}$, then 
\[
\begin{split}
k^3 &= \lim_n \big(\sum_i \theta_{\xi_{i}^{n} , \eta_i^n}\big)k\big(\sum_j \theta_{\xi_{j}^{n} , \eta_j^n}\big) \\
&=\lim_n \sum_{i,j} \theta_{\xi_i\sca{\eta_i, K\xi_j},\eta_j} \in K(XJ)
\end{split}
\]
as desired.
\end{proof}

The previous lemma gives easily the following useful application.

\begin{corollary} \label{kernel}
Let $(X, A)$, $(Y, A)$ be $\ca$-correspondences and let 
\[
\begin{split}
\phi&:A \longrightarrow L(Y) \\
\phi_*&: L(X)\longrightarrow L(X\otimes_{\phi}Y); s \longmapsto s\otimes I
\end{split}
\]
Assume that $k \in K(X)$. Then $k \in \ker \phi_*$ if and only if $k \in K(X\ker\phi)$.
\end{corollary}

In particular 

\begin{corollary} \label{finally}
Let $(X, A) $ be a $\ca$-correspondence. Then for each $n \in \bbN$, the restricted map
\begin{equation} \label{somemap}
K(X^{\otimes n-1}J_X) \ni k \longmapsto k\otimes \id \in L(X^{\otimes n})
\end{equation}
is isometric.
\end{corollary}

\begin{proof} By the previous corollary, if $k$ belongs to the kernel of (\ref{somemap}), then 
\[
\sca{k\overline{\xi}, \overline{\eta}} \in \ker\phi, \mbox{ for all } \overline{\xi}, \overline{\eta} \in X^{\otimes n-1}.
\]
On the other hand $k \in K(XJ_X)$ and so Lemma~\ref{inner} shows that
\[
\sca{k\overline{\xi}, \overline{\eta}} \in J_X\subseteq \ker\phi^{\perp},\mbox{ for all } \overline{\xi}, \overline{\eta} \in X^{\otimes n-1}.
\]
Hence $k =0$.
\end{proof}

Let us prove now that $(X, A)$ admits a non-trivial covariant representation. Note that for the Fock representation we have
\[
\pi_{\infty}(a) -\psi_t\big(\phi(a)\big)= L_a\oplus 0 \oplus 0 \oplus \dots, \quad a \in A
\]
Therefore if we want a covariant representation we need to somehow mod out with the ideal generated by the above differences, with $a$ ranging over $J_X$. 
Let us try to determine that ideal. 

Note that the above differences belong to $K(\F_X)$. Now if $\overline{\xi} \in X^{\otimes m}$ and $\overline{\eta} \in X^{\otimes n}$ and $a \in J_X$, then 
\[
\begin{split}
t_{\infty}(\overline{\xi})(L_a\oplus 0 \oplus 0 \oplus &\dots)t_{\infty}(\overline{\eta})^* \\
&=\theta_{(0, 0, \dots, \overline{\xi},\dots )a, (0, 0, \dots, \overline{\eta}a, \dots)} \in K(\F_XJ_X).
\end{split}
\]
This calculation shows that the right candidate is $K(\F_XJ_X)$. (Also note that the above implies that $K(\F_XJ_X) \subseteq \ca(\pi_{\infty}, t_{\infty})$.) If we mod out by $ K(\F_XJ_X)$ we have a covariant representation $(\overline{\pi}_{\infty}, \overline{t}_{\infty})$ with 
\[
\begin{split}
\overline{\pi}_{\infty} &:A\xrightarrow{\phantom{X}\pi_{\infty}\phantom{X}}\ca(\pi_{\infty}, t_{\infty}) \xrightarrow{\phantom{XX} q\phantom{XX}}\ca(\pi_{\infty}, t_{\infty})  / K(\F_XJ_X)\\
\overline{t}_{\infty} &:X\xrightarrow{\phantom{X}\pi_{\infty}\phantom{X}}\ca(\pi_{\infty}, t_{\infty}) \xrightarrow{\phantom{XX} q\phantom{XX}}\ca(\pi_{\infty}, t_{\infty})  / K(\F_XJ_X),
\end{split}
\]
where $q$ denotes the quotient map. In order to show that $(\overline{\pi}_{\infty}, \overline{t}_{\infty})$ is non-trivial, we verify that it is actually injective.

Assume that $\overline{\pi}_{\infty}(a)=0$ and so $\pi_{\infty}(a) \in K(\F_XJ_X)$ Then,
\begin{equation} \label{lots}
\begin{split}
a &\in J_X\\
\phi(a) & \in K(XJ_X) \\
\phi(a) \otimes \id &\in K(X^{\otimes 2}J_X)\\
&\vdots
\end{split}
\end{equation}
and also $\lim_n \| \phi(a)\otimes\id^n\|=0$, because $\pi_{\infty}(a)$ is compact. On the other hand, Corollary~\ref{finally} implies that,
\begin{equation} \label{isometries}
\|\phi(a)\|=\|\phi(a)\otimes \id\|= \|\phi(a)\otimes\id^2\|=\dots
\end{equation}
and so $\phi(a) =0$. Since $a \in J_X\subseteq \ker\phi^{\perp}$ we obtain $a =0$. This shows that $(\overline{\pi}_{\infty}, \overline{t}_{\infty})$ is injective and thus non-trivial.

\subsection{The gauge-invariance uniqueness Theorems} Having established that the representation theory of  the Cuntz-Pimsner and Cuntz-Pimsner-Toeplitz algebras is not in vacuum, now we need a test to let us know when a particular representation of these algebras is actually faithful.

\begin{definition} A representation $(\pi , t)$ of $X$ is said to admit a gauge action if for each $z \in \bbT$ there exists a $*$-homomorphism $$\beta_z \colon \ca(\pi , t )\rightarrow \ca(\pi , t)$$ such that $\beta_z (\pi(a))=\pi(a)$ and $\beta_z (t(\xi))= zt(\xi)$, for all $a \in A$ and $\xi \in X$.\end{definition}

\begin{theorem}[Katsura \cite{Katsura}, 2004] \label{giu1} Let $(X, A, \phi)$ be a $\ca$-cor\break respondence and $(\pi , t)$ a covariant representation that admits a gauge action and is faithful on $A$. Then the integrated representation $\rho = \pi \times t$ is faithful on $\O_X$.\end{theorem}
\begin{proof}
I will sketch the proof only in the case where $J_X = A$, i.e., $\phi$ is injective and $\phi(A) \subseteq K(X)$.

Let $(\overline{\pi}_{\infty}, \overline{t}_{\infty})$ be the universal covariant representation. Since both $\pi$ and $\overline{\pi}_{\infty}$ are injective, the maps $\psi_{t^n}$ and $\psi_{\overline{t}_{\infty}^n}$, $n \in \bbN$, are injective. Let
\[
A_n \equiv \psi_{t^n}\big(K(X^{\otimes n})\big) \mbox{ and }\bar{A}_n\equiv \psi_{\overline{t}_{\infty}^n} \big(K(X^{\otimes n})\big) , \quad n \in \bbN.
\]
Since $A = J_X$ we have $\pi(A) \subseteq \psi_t (K(X))$ and inductively $A_n \subseteq A_{n+1}$, $n \in \bbN$. Similarly $\bar{A}_n \subseteq \bar{A}_{n+1}$, $n \in \bbN$. Set $A_{\infty}= \overline{\cup_{n=1}^{\infty} A_n}$ and similarly for $\bar{A}_{\infty}$. Clearly $A_{\infty}$ and $\bar{A}_{\infty}$ are the fixed point algebras for the gauge actions on $ \ca(\pi , t )$ and $O_X$ respectively. Since $\rho\mid_{\bar{A}_n} =\psi_{t^n} \circ \psi_{\overline{t}_{\infty}^n}^{-1}$, $n \in \bbN$, we obtain that $\rho\mid_{\bar{A}_n} $ is injective. Since $\bar{A}_{\infty}$ is an ascending union, $\rho\mid_{\bar{A}_{\infty}} $ is also injective. This suffices to prove the injectivity of $\rho$.

Indeed assume that there exists a positive $x \in O_X \cap \ker \rho$. Since $\rho$ intertwines the gauge actions, i.e., $\rho \circ \bar{\beta}_z = \beta_z\circ \rho$, $z \in \bbT$, we have
\begin{equation} \label{eq;Katsura1}
\rho \circ \bar{\Phi} = \Phi \circ \rho,
\end{equation}
where $\Phi$ is the faithful expectation on $\ca(\pi , t)$ projecting on $A_{\infty}$, i.e., $\Phi(s) = \int \beta_z(s)dz $, $s \in \ca(\pi, t)$, and similarly for $\bar{\Phi}$. Applying (\ref{eq;Katsura1}) on $x$ we get that $\rho\big( \bar{\Phi}(x) \big) = 0$ and so $\bar{\Phi}(x) \in \ker \rho \cap \bar{A}_{\infty}=\{0\}$, as desired.
\end{proof}

Combining the above theorem with our earlier results, we obtain that the representation $\overline{\pi}_{\infty} \times \overline{t}_{\infty}$ is faithful for $\O_X$, as $(\overline{\pi}_{\infty}, \overline{t}_{\infty})$ is injective.

\begin{theorem}[Katsura \cite{Katsura}, 2004] \label{thm;giu2}
Let $(X, A, \phi)$ be a $\ca$-cor-\break respondence and $(\pi , t)$ a representation that admits a gauge action and satisfies
\[
I'(\pi, t)\equiv \{ a \in A \mid \pi(a) \in \psi_t(K(X)) \} = {0}
\]
Then the integrated representation $\pi \times t$ is faithful on $\T_X$.\end{theorem}

Using the above theorem one can easily see that if $( \pi_{\infty}, t_{\infty})$ is the Fock representation, then $ \pi_{\infty} \times t_{\infty}$ is faithful for $\T_X$.

Let us give an application of the gauge invariance uniqueness theorem to tensor algebras. We need the following result which generalizes the well known fact that the restriction of the Calkin map on the algebra generated by the shift is an isometry. (Under the additional assumption that $X$ is strict, this result was obtained by Muhly and Solel \cite{MS}.)

\begin{proposition}[Katsoulis and Kribs \cite{KatsoulisKribsJFA}, 2006] If $(X, A, \phi)$ is an injective correspondence, then
$$  \alg ({\pi}_{\infty}, {t}_{\infty}) / K(\F_X)  \simeq
 \alg ({\pi}_{\infty}, {t}_{\infty})
$$
\end{proposition}

The following results tell us that for an injective correspondence $(X, A)$, the tensor algebras $\T_X^+$ sits naturally inside $\O_X$.  It generalizes the fact that the operator algebra generated by the shift operator is completely isometrically isomorphic with the disc algebra and therefore sits inside $C(\bbT)$.

\begin{corollary} \label{MSKK} If $(X, A)$ is an injective correspondence, then $\T_X^+$ embeds isometrically and canonically in $\O_X$.
\end{corollary}

\begin{proof}
In the previous proposition we saw that $ \alg ({\pi}_{\infty}, {t}_{\infty}) / K(\F_X)  \simeq
 \alg ({\pi}_{\infty}, {t}_{\infty})$ and so $ \alg ({\pi}_{\infty}, {t}_{\infty}) / K(\F_{XJ_X})  \simeq
 \alg ({\pi}_{\infty}, {t}_{\infty})$. However, as we commented right after the proof of Theorem~\ref{giu1}, $$ \alg ({\pi}_{\infty}, {t}_{\infty}) / K(\F_{XJ_X}) \subseteq \ca({\pi}_{\infty}, {t}_{\infty}) / K(\F_{XJ_X})  \simeq \O_X$$ and we are done.
\end{proof}

There is a stronger formulation for the above result:  if $(X, A)$ is an injective correspondence, then the $\ca$-envelope of  $\T_X^+$ is isomorphic to $\O_X$. The reader can go now to Section 5 for the appropriate definitions or even a proof of that result. Note however that the assumption of injectivity for $(X, A)$ cannot be removed without the considerations of the section that follows.


\section{Adding tails to a $\ca$-correspondence}

By adding a tail to a $\ca$-correspondence, one can study arbitrary $\ca$-correspondences with the aid of  injective ones, which in general are better behaved. Loosely speaking, we say that a $\ca$-correspondence $(Y, B, \psi)$ arises from $(X, A, \phi)$ by adding a tail iff
\begin{itemize}
\item[(i)] $X \subseteq Y$ and $A \subseteq B$, with
$$\psi(a)\xi = \phi(a) \xi, \quad a \in A, \xi \in X$$
\item[(ii)] a covariant representation of $(Y, B, \psi)$ restricts to a covariant representation of  $(X, A, \phi)$
\item[(iii)] $\O_X$ is a full corner of $\O_Y$
\end{itemize}

The origins of this concept are in the theory of graph $\ca$-algebras.

Let $\G$ be a connected, directed graph with a distinguished sink $p_0 \in \G^0$ and no sources. We assume that $\G$ is contractible at $p_0$, i.e., there exists a unique infinite path $w_0 = e_1 e_2 e_3\dots$ ending at $p_0$, i.e. $r(w_0)=p_0$ and the saturation of  p is the whole graph. (Saturation in the sense of \cite{Bates}). Let $p_n\equiv s(e_n)$, $n \geq 1$.

Let $(A_p)_{p \in \G^0}$ be a family of $\ca$-algebras parameterized by the vertices of $\G$ so that $A_{p_{0}}=A$. For each $e \in \G^1$, we now consider a full, right Hilbert $A_{s(e)}$~-~module $X_e$ and a $*$-homomorphism
\[
\phi_e \colon A_{r(e)} \longrightarrow \L(X_e)
\]
satisfying the following requirements.

\begin{itemize}

\item If $e \neq e_1$, $\phi_e$ is injective and maps onto $\K(X_e)$.

\item $\K(X_{e_1}) \subseteq \phi_{e_1} (A)$ and
\begin{equation} \label{inj}
J_{X}\subseteq \ker \phi_{e_1} \subseteq \left( \ker \phi_X
\right)^{\perp}.
\end{equation}
\item The maps $\phi_{X}$ and $\phi_{e_1}$ satisfy the \textit{linking condition}
\begin{equation} \label{linkin}
\phi_{e_1}^{-1}(\K(X_{e_1})) \subseteq \phi_X^{-1}(\K(X))
\end{equation}
\end{itemize}

Let $$T_0= c_0 (\, (A_p)_{p \in \G^{0}_{-}}),$$  where $\G^{0}_{-} \equiv \G^{0} \backslash \{p_0\}$.

Let $T_1$ be the completion of $c_{00}((X_e)_{e \in \G^{1}}) $ with respect to the inner product

\[
\sca{ u, v  }(p)= \sum_{s(e)=p} \,\sca{ u_e, v_e  }, \quad p \in \G^{0}_{-}.
\]

Equip now $T_1$ with a right $T_0$~-~action, so that
\[
(ux)_e = u_ex_{s(e)}, \quad e \in \G^{1}, x \in T_0.
\]

The pair $(T_0, T_1)$ is the \underline{tail} for $(X, A, \phi)$.

To the $\ca$-correspondence $(X, A, \phi)$ and the data
\[
\tau\equiv \Big(\G, (X_e)_{e \in \G^{1}}, (A_p)_{p\in \G^{0}},
(\phi_{e})_{e \in \G^{1}} \Big),
\]
we now associate
\begin{align} \label{tau}
\begin{split}
\atau &\equiv A\oplus T_0  \\
\xtau &\equiv X \oplus  T_1
\end{split}
\end{align}
and we view $\xtau$ as a $\atau$-Hilbert module.

We define a left $\atau$-action $\ptau: \atau \rightarrow \L(\xtau)$ on $\xtau$ by setting
\[
\ptau(a, x \, )(\xi, u )= (\phi_X(a)\xi, v  ),
\]
where
\[
v_{e}= \left\{ \begin{array}{ll}
   \phi_{e_1}(a)(u_{e_1}),& \mbox{if $e = e_1$} \\
   \phi_e (x_{r(e)})u_e, & \mbox{otherwise}
   \end{array}
\right.
\]
for $a \in A$, $\xi \in X$, $x \in T_0$ and $u \in T_1$.

\begin{theorem}[Kakariadis and Katsoulis \cite{KakKatTrans}, 2012] \label{KKtails}
Let $(X, A, \phi)$ be a non-injective C*- correspondence and let $\xtau$ be the graph $\ca$-cor-\break respondence over $\atau$ defined above. Then $\xtau$ is an injective $\ca$-cor-\break respondence and the Cuntz-Pimsner algebra $\O_X$ is a full corner of $\O_{\xtau}$.\end{theorem}

Furthermore, if $(\pi , t)$ is a covariant representation of $\xtau$, then its restriction on $X$ produces a covariant representation of $(X, A, \phi)$ .

\subsection{The Muhly-Tomforde tail} Our Theorem \ref{KKtails} was inspired by relevant work of Muhly and Tomforde \cite{MT}. Given a (non-injective) correspondence $(X, A, \phi_X)$, Muhly and Tomforde construct the tail that results from the previous construction, with respect to data
\[
\tau =\Big(\G, (X_e)_{e \in \edg}, (A_p)_{p\in \vrt}, (\phi_{e})_{e
\in \edg} \Big)
\]
defined as follows.

The graph $\G$ is illustrated in the figure below.

\vspace{.2in}
\[
\xymatrix{&{\bullet^{p_0}} &{\,\, \bullet}^{p_1}
\ar[l]^{e_1}&{\,\,\bullet^{p_2}} \ar[l]^{e_2} &{\bullet^{p_3}}
\ar[l]^{e_3}&{\,\,\bullet} \ar[l]&\dots \ar[l]}
\]
\vspace{.2in}

\noindent $A_p=X_e=\ker \phi_X$, for all $p \in \vrtm$ and $e \in \edg$. Finally,
\[
\phi_e(a)u_e=au_e, \quad e \in \edg , u_e \in X_e , a \in A_{r(e)}
\]

\subsection{The tail for $(A, A, \alpha)$} Given a (non-injective) correspondence $(X, A, \phi_X)$, we construct the tail that results from the previous construction, with respect to data
\[
\tau =\Big(\G, (X_e)_{e \in \edg}, (A_p)_{p\in \vrt}, (\phi_{e})_{e
\in \edg} \Big)
\]
defined as follows.

Let $\theta :A\rightarrow M(\ker \phi_X)$.

The graph $\G$ is once again

\vspace{.2in}
\[
\xymatrix{&{\bullet^{p_0}} &{\,\, \bullet}^{p_1}
\ar[l]^{e_1}&{\,\,\bullet^{p_2}} \ar[l]^{e_2} &{\bullet^{p_3}}
\ar[l]^{e_3}&{\,\,\bullet} \ar[l]&\dots \ar[l]}
\]
\vspace{.2in}

\noindent but $A_p=X_e=\theta (A)$, for all $p \in \vrtm$ and $e \in \edg$. Finally,
\[
\phi_e(a)u_e=\theta(a)u_e, \quad e \in \edg , u_e \in X_e , a \in A_{r(e)}
\]

Using the technique of ``adding tails" we can dispose of the injectivity assumption in Corollary~\ref{MSKK}.

\begin{theorem} [Katsoulis and Kribs \cite{KatsoulisKribsJFA}, 2006] If $(X, A, \phi)$ is any $\ca$-correspondence, then $\T_X^+$ embeds isometrically and canonically in $\O_X$.
\end{theorem}

Another application of adding tails appears in \cite{KakKatTrans}.

\begin{theorem} Let $(A, \alpha)$ a $\ca$-dynamical system and $X_{\alpha}$ the pertinent correspondence. Then the Cuntz-Pimsner $\ca$-algebra $O_{X_{\alpha}}$ is strongly Morita equivalent to a crossed product $\ca$-algebra.\end{theorem}

Finally let us give a sample of how exciting things can get with this process of ``adding tails". This material is not required for accessing the rest of the paper.

\begin{definition} \label{multdefn}
 A \textit{multivariable $\ca$-dynamical system}
is a pair $( A, \ga)$ consisting of
a $\ca$-algebra $A$ along with
a tuple $\ga=( \ga_1, \ga_2 , \dots, \ga_n)$, $n \in \bbN$, of $*$-endomorphisms of $A$. The dynamical system is called injective iff
$\cap_{i=1}^n \, \ker\ga_i =\{0\}$.
To the multivariable system $( A, \ga)$ we associate a $\ca$-correspondence $(X_{\ga}, A, \phi_{\ga})$ as follows. Let $X_{\ga}=A^n = \oplus_{i=1}^n A$ be the usual right $A$-module. That is
\begin{enumerate}
\item $(a_1,\dots,a_n)\cdot a= (a_1 a ,\dots,a_n a)$,
\item $\sca{(a_1,\dots,a_n), (b_1,\dots,b_n)}=\sum_{i=1}^n
\sca{a_i,b_i}=\sum_{i=1}^n a_i^*b_i$.
\end{enumerate}
Also, by defining the $*$-homomorphism
\begin{align*}
\phi_{\ga}\colon A \longrightarrow \L(X_{\ga})\colon a \longmapsto \oplus_{i=1}^n \ga_i(a),
\end{align*}
$X$ becomes a $\ca$-correspondence over $A$, with $\ker\phi_{\ga}=
\cap_{i=1}^n \ker \ga_i$ and $\phi(A) \subseteq \K(X_{\ga})$. 
\end{definition}

It is
easy to check that in the case where $A$ and all $\ga_i$ are unital, $X$ is finitely generated as an
 $A$-module by the elements
\[
e_1:=(1,0,\dots,0),e_2:=(0,1,\dots,0),\dots,e_n:=(0,0,\dots,1),
\]
where $1\equiv 1_A$. In that case, $(\pi,t)$ is a representation of this
$\ca$-correspondence if, and only if, the $t(\xi_i)$'s are
isometries with pairwise orthogonal ranges and
\[
\pi(c)t(\xi)=t(\xi)\pi(\ga_i(c)), \quad i=1,\dots, n.
\]

\begin{definition}
The Cuntz-Pimsner algebra $\O_{(A, \ga)}$ of a multivariable $\ca$-dynamical system
$(A,\ga)$ is the Cuntz-Pimsner algebra of the
$\ca$-correspondence $(X_{\ga}, A, \phi_{\ga})$ constructed as above
\end{definition}

In the $\ca$-algebra literature, the algebras $\O_{(A, \ga)}$ are denoted as \break $A \times_{\ga} \O_n$ and go by the name ``twisted tensor products by $\O_n$". They were first introduced and studied by Cuntz~\cite{Cun} in 1981. In the non-selfadjoint literature, there algebras are much more recent. In Section 6, we will see the tensor algebra $\T^+{(A, \ga)}$ for a  multivariable dynamical system $(A,\ga)$. It turns out that $\T^+{(A, \ga)}$ is completely isometrically isomorphic to the tensor algebra for the $\ca$-correspondence $(X_{\ga}, A, \phi_{\ga})$. As such, $\O_{(A, \ga)}$ is the $\ca$-envelope of $\T^+{(A, \ga)}$. Therefore, $\O_{(A, \ga)}$ provides a very important invariant for the study of isomorphisms between the tensor algebras $\T_{(A, \ga)}$.

We now apply our ``adding tails" technique 
to the
$\ca$-correspondence defined above. The graph $\G$ that we associate with $(X_{\ga}, A, \phi_{\ga})$ has no loop edges and a single sink $p_0$. All vertices in $\vrt \backslash \{p_0\}$ emit $n$ edges, i.e., as many as the maps involved in the multivariable system, and receive exactly one. In the case where $n=2$, the following figure illustrates $\G$.
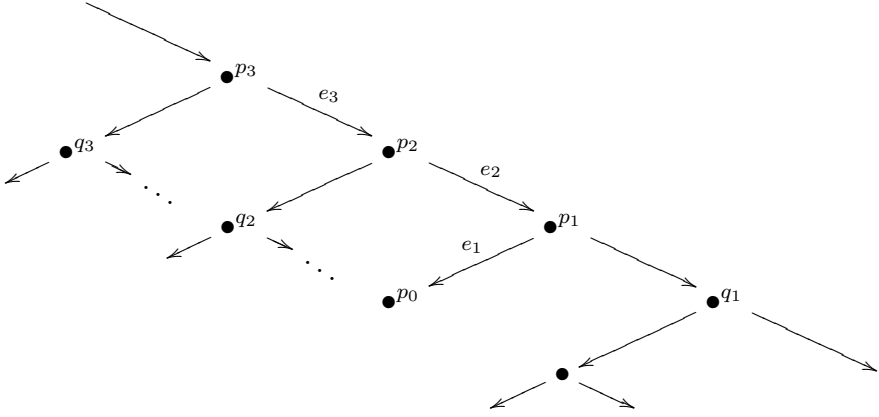
\vspace{.2in}
\[
\xygraph{
!{<0cm,0cm>;<2.15cm,0cm>:<0cm,1cm>::}
!{(1,3)}*+{\empty}="a"
!{(2,2)}*+{\bullet^{p_3}}="b"
!{(1,1)}*+{\bullet^{q_3}}="c"
!{(3,1)}*+{\bullet^{p_2}}="cc"
!{(0.5,.5)}*+{\empty}="e"
!{(1.5,.5)}*+{\ddots}="f"
!{(4,0)}*+{\bullet^{p_1}}="j"
!{(5,-1)}*+{\bullet^{q_1}}="q"
!{(6,-2)}*+{\empty}="u"
!{(7,-3)}*+{\empty}="bb"
!{(3,-1)}*+{\bullet^{p_0}}="p"
!{(4,-2)}*+{\bullet}="w"
!{(3.5,-2.5)}*+{\empty}="v"
!{(4.5,-2.5)}*+{\empty}="y"
!{(2,0)}*+{\bullet^{q_2}}="i"
!{(1.5,-.5)}*+{\empty}="k"
!{(2.5,-.5)}*+{\ddots}="l"
"b" :"cc"^{e_3}
"cc":"j"^{e_2}
"j":"q"
"j":"p"_{e_1}
"a":"b"
"q":"w"
"q":"u"
"cc":"i"
"b":"c"
"c":"e" "c":"f"
"i":"k"
"i":"l"
"w":"v"
"w":"y"
}
\]
\vspace{.2in}

\noindent There is a unique infinite path $w$ ending at $p_0$ whose saturation is the whole graph and so the requirements of Theorem~\ref{KKtails} are satisfied, i.e., $\G$ is contractible at $p_0$.

Let $\J \equiv \cap_{i=1}^{n} \, \ker \alpha_i$ and let $M(\J)$ be the multiplier algebra of $\J$. Let $\theta \colon A \longrightarrow M(\J)$ the map that extends the natural inclusion $\J \subseteq M(\J))$. Let $X_e= A_{s(e)} =\theta(A)$, for all $e \in \edg$, and consider $(X_e, A_{s(e)})$ with the natural structure that makes it into a right Hilbert module.

For $e \in \edg \backslash \{e_1\}$ we define $\phi_{e}(a)$ as left multiplication by $a$. With that left action, clearly $X_e$ becomes an $A_{r(e)}-A_{s(e)}$-equivalence bimodule. For $e = e_1$, it is easy to see that
\[
\phi_{e_1}(a)(\theta(b))\equiv \theta(ab), \quad a,b \in A
\]
defines a left action on $X_{e_1}=\theta(A)$, which satisfies both (\ref{inj}) and (\ref{linkin}).

 For the $\ca$-correspondence $(X_{\ga}, A, \phi_{\ga})$ and the data
  \[
 \tau = \Big(\G, (X_e)_{e \in \edg}, (A_p)_{p\in \vrt}, (\phi_{e})_{e \in \edg} \Big),
   \]
   we now let $((X_{\ga})_{\tau} , A_{\tau}, (\phi_{\ga})_{\tau}) $ be the $\ca$-correspondence constructed as in the previous section. For notational simplicity $((X_{\ga})_{\tau} , A_{\tau}, (\phi_{\ga})_{\tau}) $ will be denoted as $(\xtau, \atau, \ptau)$. Therefore
\begin{align*}
\atau &= A\oplus c_0 (\vrtm, \theta(A))  \\
\xtau &= A^n \oplus  c_0 (\G^{(1)} ,\theta(A)).
\end{align*}
Now label the n-edges of $\G$ emitting from each $p \in \vrtm$ as $p^{(1)}, p^{(2)},\break  \dots,p^{(n)}$. It is easy to see now that the mapping
\[
c_0 (\G^{(1)},\theta(A)) \ni u \longmapsto \oplus_{i=1}^n \{ u(p^{(i)})\}_{p \in \vrt} \in \oplus_{i=1}^n \, c_0 (\vrtm,\theta(A))
\]
establishes a unitary equivalence
\begin{align*}
\xtau &=  A^n \oplus  c_0 (\G^{(1)} ,\theta(A)) \\
        &\cong  \oplus_{i=1}^n \, \left( A\oplus c_0 (\vrtm,\theta(A)\right)
\end{align*}
between the Hilbert $A$-module $\xtau$ and the $n$-fold direct sum of the $\ca$-algebra
$ A\oplus c_0 (\vrtm,\theta(A))$, equipped with the usual $A \oplus c_0 (\vrtm,\theta(A))$-right action and inner product.

It only remains to show that the left action on $\xtau$ comes from an $n$-tuple of $*$-endomorphisms of
$A\oplus c_0 (\vrtm,\theta(A))$. This is established as follows.

For any $i=1, 2, \dots, n$ and  $(a, x) \in A\oplus c_0 (\vrtm,\theta(A))$ we define
\begin{equation*}
\hat{\ga}_i (a, x)=( \ga_i (a), \gamma_i(a,x))
\end{equation*}
where $\gamma_i(a,x) \in c_0 (\vrtm,\theta(A))$ with
\[
\gamma_i(a,x)(p)=
\left\{
\begin{array}{ll}
\theta(a), &\mbox{if } p^{(i)}=e_0, \\
x(r(p^{(i)})), &\mbox{otherwise.}
\end{array}
\right.
\]
It is easy to see now that $\left( A\oplus c_0 (\vrtm,\theta(A)), \hat{\ga}_1, \dots,\hat{\ga}_n \right)$ is a multivariable dynamical system, so that the $\ca$-correspondence associated with it is unitarily equivalent to
$(\xtau, \atau, \ptau)$.

We have therefore proved

\begin{theorem} \label{multitail}
If $(A, \ga)$ is a non-injective multivariable $\ca$-dynamical system, then there exists an injective multivariable $\ca$-dynamical system $(B, \beta)$ so that the associated Cuntz-Pimsner algebras $\O_{(A, \ga)}$ is a full corner of
$\O_{(B, \beta)}$. Moreover, if $A$ belongs to a class
$\, \C$ of $\, \ca$-algebras which is invariant under quotients and $c_0$-sums, then $B\in \C$ as well. Furthermore, if $(A, \ga)$ is non-degenerate, then so is $(B, \beta)$.
\end{theorem}


\section{The $\ca$-envelope of an operator algebra}

Most experts will agree that the concept of the $\ca$-envelope is at the heart of the modern operator algebra theory. This is one of the lasting contributions of Bill Arveson \cite{Arvenv, Arvsub} that will keep us busy for years to come. The presentation we give here is complete with proofs and it is essentially the Dritschel and McCullough approach to the subject \cite{DrMc} via maximal dilations.

What is the $\ca$-envelope of an operator space $\S$? There are more than one ways to approach the answer. Some might opt for the categorical approach: the $\ca$-envelope is the ``smallest" $\ca$-algebra containing $\S$. (See the statement  of Theorem \ref{smallestC}.) Others, like myself, prefer the ``utility grade" approach: the $\ca$-envelope is the $\ca$-algebra generated by the range of any maximal and isometric representation of $\S$. (Now \textit{read} the proof of Theorem \ref{smallestC}.) The truth be told, the $\ca$-envelope is an elusive object as we only know of its existence through non-constructive proofs. As you can imagine, identifying the $\ca$-envelope, even for very concrete algebras or spaces, can be quite a feat.

In this section, all operator spaces $\S$ satisfy $1 \in \S \subseteq \ca(\S)$ and all completely contractive maps between operator spaces preserve the unit. See the monographs \cite{BlLM, Paulsen} for the basic definitions and results, such as the one appearing below.

\begin{theorem} [Arveson \cite{Arvsub}, 1969] A (unital) completely contractive map $\phi:\S \rightarrow B(\H)$ admits a completely contractive (unital) extension
\[
\tilde{\phi}: \ca (\S) \longrightarrow B(\H).
\]
\end{theorem}

\begin{definition} A completely contractive (cc) map $\phi:\S \rightarrow B(\H)$ is said to have the \textit{unique extension property} iff any completely contractive extension
\[
\tilde{\phi}: \ca (\S) \longrightarrow B(\H)
\]
is multiplicative.
\end{definition}

\begin{definition} If $\phi_i:\S \rightarrow B(\H_i)$, $i=1,2$, are cc maps then $\phi_2$ is said to be a dilation of $\phi_1$ (denoted as $\phi_2\geq\phi_1$) if $\H_2\supseteq \H_1$ and
\[
c_{\H_1}(\phi_2(s)) \equiv P_{\H_1}\phi_2(s) \mid_{\H_1} =\phi_1(s),\quad \forall s \in \S.
\]

\end{definition}

\begin{definition} A completely contractive (cc) map $\phi:\S \rightarrow B(\H)$ is said to be \textit{maximal} if it has no non-trivial dilations: $\phi'\geq \phi \implies \phi'=\phi\oplus \psi$ for some cc map $\psi$.
\end{definition}

\begin{theorem} [Muhly and Solel \cite{MS0}, 1998] \label{MSthm} A completely contractive map $\phi: \S \rightarrow B(\H)$ is maximal iff it has the unique extension property.\end{theorem}

\begin{proof}
Suppose $\phi$ is maximal and $\tilde{\phi}$ a cc extension on $\ca(\S)$. By Stinespring's Theorem there exists a dilation $\rho$ of $\phi$ so that the following diagram
\[
\xymatrix{&B(\K) \ar[d]^{c} \\
\ca(\S) \ar[ur]^{\rho} \ar[r]_{\tilde{\phi}}  & B(\H)}
\]
commutes. (Here $c$ denotes compression on $\H$.) The map $\rho$ is a $*$-representation and we may assume that
\[
\big[ \rho\big( \ca(\S)\big)(\H)\big]=\K
\]
or otherwise we compress.

Since $\tilde{\phi}$ is maximal, for any $s \in \S$ we have $\rho(s)=\phi(s)\oplus s_1$
for some $s_1$ and so $\rho(s^*)=\rho(s)^*=\phi(s)^*\oplus s_1^*$.
Hence \[
\big[ \rho\big( \ca(\S)\big)(\H)\big]=\H
\]
and so $\K = \H$ and $c=\id$. Therefore $\tilde{\phi} = \rho$ is multiplicative.

Conversely, assume that $\phi: \S \rightarrow B(\H)$ has the unique extension property and let $\rho: \S \rightarrow B(\K)$ be a dilation of $\phi$. Extend both $\phi$ and $\rho$ to completely contractive maps $\tilde{\phi}$ and $\tilde{\rho}$ so that the diagram
\[
\xymatrix{&B(\K) \ar[d]^{c} \\
\ca(\S) \ar[ur]^{\tilde{\rho}} \ar[r]_{\tilde{\phi}}  & B(\H)}
\]
commutes on $\S$. Hence the completely contractive map $c \circ \tilde{\rho}$ agrees with $\phi$ on $\S$ and since $\phi$ has the unique extension property, $c \circ \tilde{\rho}$ is multiplicative. Hence
\[
 P_{\H}\tilde{\rho}(s^*s)P_{\H}= P_{\H}\tilde{\rho}(s^*)P_{\H}\tilde{\rho}(s)P_{\H}.
\]
Also by the Swarchz inequality
\[
P_{\H}\tilde{\rho}(s^*s)P_{\H} \geq P_{\H}\tilde{\rho}(s^*) \tilde{\rho}(s)P_{\H}.
\]
Subtracting the above gives $$0 \geq P_{\H}\tilde{\rho}(s^*)(I -P_{\H})\tilde{\rho}(s)P_{\H}= \big( (I -P_{\H})\tilde{\rho}(s)P_{\H}\big)^*\big( (I -P_{\H})\tilde{\rho}(s)P_{\H} \big) \geq0$$
and so $\tilde{\rho}((\S)$ leaves $\H$ invariant. Analogous calculations with $\tilde{\rho}(ss^*)$ also imply invariance of $\H$ by $\tilde{\rho}(\S^*)=\tilde{\rho}(\S)^*$ and so $\H$ reduces $\tilde{\rho}(\S)$ and therefore $\rho(\S)$. Hence $\rho$ is a trivial dilation.
\end{proof}

As I explained in the introduction, the range of a completely isometric and maximal map will give us the $\ca$-envelope. Proving that such maximal maps do exist requires a clever trick. 

\begin{theorem} [Dritschel and McCullough \cite{DrMc}, 2005] Every cc map $\phi: \S \rightarrow B(\H)$ can be dilated to a maximal cc map  $\phi': \S \rightarrow B(\H')$.
\end{theorem}

\begin{proof} For convenience we assume that both $\A$ and $\H$ are separable. The proof proceeds in two steps. First we prove that $\phi$ admits a dilation $\psi$ on a Hilbert space $\H_{\psi}$ which is \textit{maximal with respect to} $\phi$. That means that any dilation $\psi'$ of $\psi$ on some Hilbert space $\H_{\psi'}$ satisfies
\begin{equation} \label{failure}
\big(P_{\H_{\psi'}}-P_{\H_{\psi}}\big)\psi'P_{\H}=P_{\H}\psi \big(P_{\H_{\psi'}}-P_{\H_{\psi}}\big) =0.
\end{equation}
By way of contradiction assume that such a dilation does not exist for $\phi$. Hence there exists a dilation $\psi_1$ for $\phi$ which is non-trivial, i.e., it does not reduce $\H$. Since $\psi_1$ is not maximal with respect to $\phi$, it admits a dilation $\psi_2$ that fails~(\ref{failure}), i.e.,
\begin{equation*}
\big(P_{\H_{\psi_2}}-P_{\H_{\psi_1}}\big)\psi'P_{\H}=P_{\H}\psi \big(P_{\H_{\psi_2}}-P_{\H_{\psi_1}}\big) =0.
\end{equation*}
Since $\psi_2$ is not maximal with respect to $\phi$, it admits a dilation $\psi_3$ that once again fails~(\ref{failure}), i.e.,
\begin{equation*}
\big(P_{\H_{\psi_3}}-P_{\H_{\psi_2}}\big)\psi'P_{\H}=P_{\H}\psi \big(P_{\H_{\psi_3}}-P_{\H_{\psi_2}}\big) =0.
\end{equation*}
Continuing like this and using transfinite induction, we obtain dilations $\{ \psi_i\}_{i \in \bbI}$ that fail~(\ref{failure}) with respect to their successors, where $\bbI$ denotes the first uncountable ordinal. Since the gap ordinals are uncountable, we obtain a contradiction because of our separability assumption. 

In order to finish the proof, let $\psi_1$ be a dilation of $\phi$ which is maximal with respect to $\phi$, let $\psi_2$ be a dilation of $\psi_1$ which is maximal with respect to $\psi_1$ and so on. Any weak limit of the sequence $\{ \psi_n\}_{n \in \bbN}$ gives the desired maximal dilation $\phi'$ of $\phi$.
\end{proof}

Notice an important implication of Theorem \ref{MSthm}. If $\phi: \S \rightarrow B(\H)$ is completely contractive homomorphism of a (unital) operator algebra $\S$, then any maximal dilation of $\phi$ is automatically multiplicative. The same conclusion on multiplicativity holds for arbitrary completely isometric maps between operator algebras as we are about to see. First we need the following. 

\begin{lemma} [Arveson \cite{Arvsub}, 1969] Let $\S , \T$ be operator spaces and $
\alpha: \S \longrightarrow \T$
be a completely isometric (unital) map. If $\phi:\T \rightarrow B(\H)$ is maximal then $$\phi \circ \alpha :\S \rightarrow B(\H)$$ is also maximal.
\end{lemma}

\begin{proof}
Indeed assume that $\rho$ dilates $\phi\circ \alpha $ so that the diagram
\[
\xymatrix{& &B(\K) \ar[d]^{c} \\
\S \ar[urr]^{\rho} \ar[r]_{\alpha}  & \T \ar[r]_{\phi} & B(\H)}
\]
commutes. Then  for any $t \in \T$ we have
\[
 P_{\H}\rho\big( \alpha^{-1}(t)\big)\mid_{\H} = c \big( \rho\big( \alpha^{-1}(t)\big) \big)= \phi(t)
\]and so by the maximality of $\phi$, there exists $t_1$ so that
\[
\rho\big( \alpha^{-1}(t)\big)= \phi(t)\oplus t_1.
\]
Substituting $ t = \alpha(s)$, $ s \in \S$ in the above, we get
\[
\rho(s) = (\phi\circ \alpha)(s)\oplus s_1
\]
and so $\phi \circ\alpha$ is a maximal map.
\end{proof}

\begin{corollary} If $\A, \B$ are unital operator algebras and
\[
\alpha : \A \longrightarrow \B
\]
is a complete isometry, then $\alpha$ is multiplicative.
\end{corollary}

\begin{proof} Consider the diagram
\[\xymatrix{
\A\ar[r]^{\alpha}  & \B \ar[r]^{\rho} & B(\H)}
\]
with $\rho$ a maximal completely isometric map. Since $\rho$ is the restriction of a $*$-homomorphism on $\ca(\B)$, $\rho$ is multiplicative. By the previous lemma, $\rho \circ \alpha$ is maximal. Arguing as above, we obtain that $\rho \circ \alpha$ is also multiplicative. Hence, $\alpha = \rho^{-1}\circ( \rho \circ \alpha)$ is multiplicative.
\end{proof}

The following results were discovered by Arveson \cite{Arvenv} in special case and by Hamana \cite{Hamana} in complete generality.

\begin{theorem}   [Hamana \cite{Hamana}, 1979] \label{Hamana1}Let $\phi:\S \rightarrow B(\H)$ be a completely isometric maximal map. If $\J \subseteq \ca(\S)$ is an ideal so that the quotient map
\[
q: \ca(\S) \longrightarrow \ca(\S)/\J
\]
is faithful on $\S$, then
\[
\J \subseteq \ker \tilde{ \phi}
\]
where $\tilde{\phi}$ is the unique cc extension of $\phi$ to $\ca(\S)$. (The ideal $\ker \tilde{ \phi}$ is said to be the \textit{Shilov ideal} of $\S\subseteq \ca(\S)$.)\end{theorem}

\begin{proof} Consider a map $\theta$ that makes the following diagram
\[
\xymatrix{& \S\slash \J \ar[dr]^{\theta}& \\
\S \ar[ur]^{q} \ar[rr]_{\phi}  & & B(\H)}
\]
commutative. Now extend $\theta $ to obtain a diagram
\[
\xymatrix{& \ca(\S)\slash \J \ar[dr]^{\tilde{\theta}}& \\
\ca(\S) \ar[ur]^{q} \ar[rr]_{\tilde{\phi}}  & & B(\H)}
\]which is not a priori commutative. However, $\tilde{\theta}\circ q$ is a completely contractive extension of $\phi$, which has the unique extension property. Hence $\tilde{\phi}=\tilde{\theta}\circ q$. In particular, $\J = \ker q \subseteq \ker \tilde{\phi}$.
\end{proof}

Finally here is the existence of the $\ca$-envelope.

\begin{theorem} [Hamana \cite{Hamana}, 1979] \label{smallestC}Let $\S$ be a unital operator space. Then there exists a $\ca$-algebra $\cenv (\S)$ and a complete unital isometry
\[
j \colon \S \longrightarrow \cenv (\S)
\]
so that for any other completely isometric unital embedding
\[
\phi \colon \S \longrightarrow \C=\ca(\phi(\S))
\]
we have $*$-homomorphism $\pi \colon \C \rightarrow \cenv(\S)$ so that $\pi\circ \phi = j$.\end{theorem}

\begin{proof} The candidate for $\cenv(\S)$ is the pair $\big( \ca(j(\S)), j \big)$ where $j$ is any maximal and isometric representation $j$ of $\S$, e.g., any maximal dilation of an isometric representation of $\S$.

Indeed if
\[
\xymatrix{&\phi(\S) \subseteq \C  \\
\S\ar[ur]^{\phi} \ar[r]_{j}  & B(\H)}
\]
are as in the statement, then $j \circ \phi^{-1}$ is a maximal map and so extends to a $*$-homomorphism that makes the above diagram commutative.
\end{proof}

The pair $\big( \ca(j(\S)), j \big)$ appearing in the proof does not depend on the particular choice of the maximal and isometric map $j$, as different choices for $j$ produce ``equivalent" pairs. To make things precise, one can actually give the ``abstract" definition that \textit{$\cenv(\S) \simeq (\C, j)$ consists of a completely isometric injection $j$ in a $\ca$-algebra $\C=\ca(j(\C))$ satisfying the following: if $(\C', j')$ is any other such pair then there exists $*$-homomorphism $\pi : \C' \rightarrow \C$ so that $j = \pi\circ j'$}. Two such pairs are said to be equivalent provided that the $*$-homomorphism $\pi$ appearing above is actually isomorphism.

Having defined the $\ca$-envelope of an operator space abstractly now note that Theorem~\ref{Hamana1} implies that if $\S$ is an operator space and $\J$ the Shilov ideal of $\S$ in $\ca(\S)$ then $\cenv(\S) \simeq \big( \ca(\S)/ \J , q\big)$ is another manifestation of the $\ca$-envelope as it is conjugate to $\big( \ca(j(\S)), j \big)$ via the $*$-isomorphism $\tilde{\theta}$.

The astute reader has already realized that the $\ca$-envelope is not just the $\ca$-algebra $\cenv(\S)$ but instead the pair $(\cenv(\S), j)$. This is not to be taken lightly: a pair $(\C, j)$, where $\C \simeq \cenv(\S)$ and $j : \S \rightarrow \C=\ca(j(\C))$ is an isometry, does not automatically qualify for being $\cenv(\S)$.

 A natural question arises here: does a unital operator space admit enough maximal \textit{and irreducible} representations to capture the norm? One can weaken somewhat this question by asking irreducibility only for the extension of these maps to the $\ca$-envelope. The (affirmative) answer to this question for separable algebras was obtained by Arveson~\cite{Arvenv} and more recently by Davidson and Kennedy \cite{DavKen}, without the separability assumption. 
 
 A final remark: if $\cenv(\S) \simeq (\C, j)$ with $\C \subseteq B(\H)$, does it follow that $j: \S \rightarrow B(\H)$ is a maximal map? The answer perhaps surprisingly is no! As it turns out not every $*$-representation of $\C$ turns out to restrict to a maximal representation of $\S$. See \cite{MS} for more on that.

\subsection{The $\ca$-envelope of an arbitrary operator algebra}

If $\A \subseteq B(\H)$ is an operator algebra with $I_{\H} \notin \A$,  then $\A_1$ will denote subspace generated by adjoining $I_{\H}$ to $\A$.

\begin{theorem} [Meyer \cite{Mey}, 2001] Let $\phi \colon \A \rightarrow B(\K)$ be a completely contractive homomorphism and assume that $I_{\H} \notin \A$. Then its unitization $\phi_1 \colon \A_1 \rightarrow \B(\K)$ is also completely contractive.\end{theorem}

This result implies that for a non-degenerately acting, non-unital operator algebra $\A$, its unitization is unique: if $j: \A \rightarrow B(\K)$ is any completely isometric \textit{homomorphism} (perhaps degenerate),then $I_{\K} \neq j(\A)$ and applying the above theorem twice we obtain that $j_1$ is a complete isometry as well.  The situation becomes nicer if we restrict our attention to operator algebras with a contractive approximate unit, i.e., approximately unital operator algebras. In that case the absence of a unit for $\A$ is equivalent to the absence of the unit for any $\ca$-cover.  The same also holds for non-degeneracy. Therefore we don't have to talk about non-degenerately acting operator algebras as we make this a blanket assumption for the representations of the $\ca$-covers. 

We may consider the category of operator algebras with morphisms the completely contractive homomorphisms. In that category

\begin{corollary} Every cc homomorphism $\phi: \A \rightarrow B(\H)$ of an operator algebra $\A$ can be dilated to a maximal cc homomorphism.  $\phi': \S \rightarrow B(\H')$.
\end{corollary}

It seems though that for the above corollary to work, we need to allow for perhaps degenerate dilations as it is not clear what happens to a non degenerate dilation of a unital operator algebra when we remove the unit. This complicates things.

A different approach suggests that the $\ca$-envelope of an approximately unital operator algebra $\A$ is the $\ca$-algebra generated by $\A$ inside $\cenv(\A_1)$ with the obvious injection for $\A$. This provides the last prerequisite for reading the proof of the following result. Note that the case where $(X, A)$ is injective and strict was obtained by Muhly and Solel in \cite{MS}.

\begin{theorem} [Katsoulis and Kribs \cite{KatsoulisKribsJFA}, 2006] \label{mainthm1}
If $(X, A, \phi)$ is a $\ca$-correspondence, then
\[
(\cenv(\T^{+}_{X}) , j)\simeq \O_X
\]
via a map $j$ that sends generators to generators.
\end{theorem}

\begin{proof} I will give a proof only in the case where $A$ has a unit, which by non-degeneracy is a unit for $\O_X$ as well. By Corollary \ref{MSKK} we can view $\T^{+}_{X}$ as a canonical subalgebra of $\O_X$. It is enough to prove that $\J_{\T^{+}_{X}} =\{ 0\}$, where $\J_{\T^{+}_{X}} $ denotes the Shilov ideal of $\T^{+}_{X} \subseteq \O_X$. By way of contradiction assume otherwise.

Note that $\O_X $ admits a natural gauge action that leaves $\A$ invariant element wise and twists $X$ by unimodular scalars. Hence that gauge action leaves $\T^{+}_{X}$ invariant. Since $\J_{\T^{+}_{X}} $is the largest ideal in $\O_X$ so that the corresponding quotient map is completely isometric on $\T^{+}_{X}$, we conclude that $\J_{\T^{+}_{X}} $ is gauge invariant and so $\O_X/  \J_{\T^{+}_{X}} $ admits a gauge action. Since the natural quotient map $q: \O_X\rightarrow \O_X/  \J_{\T^{+}_{X}} $ is not faithful, Theorem \ref{giu1} implies that $q$ is not faithful on $A$ and therefore on $\T^{+}_{X}$. This is a contradiction.
\end{proof}

See \cite{KatsoulisKribsJFA} for more details.


\section{Dynamics and classification of operator algebras} \label{multivariable}

An important moment for non-selfadjoint operator algebra theory was the use of ideas from \cite{DavKatMem} in the work of Gunther Cornelissen and Matilde Marcolli  \cite{CorM}.  Cornelissen and Marcolli actually solved a problem in class field theory by making heavy use of operator algebra theory, including the theory presented in this section. Here is an outline of their result.

A complex number $a$ is called \emph{algebraic} if there exists a nonzero polynomial $p(X) \in \bbQ [X]$  such that $p(a) = 0$.
The polynomial is unique if we require that it be irreducible and monic.
We say that $a$ is an algebraic integer if the unique irreducible, monic polynomial which it satisfies has integer coefficients. We know that the set $\overline \bbQ$ of all algebraic numbers is a field, and the algebraic integers form a ring. For an algebraic number $a$, the set $K$ of all $f(a)$, with $f(X) \in \bbQ[X]$ is a field, called an algebraic number field. If all the roots of the polynomial $p(X)$ are in $K$, then $K$ is called Galois over $\bbQ$.

\begin{question}Which invariants of a number field characterize it up to isomorphism?
\end{question}

The absolute Galois group of a number field $K$ is the group $G_K = Gal (\overline{\bbQ} / K)$ consisting of all automorphisms $\sigma$ of $\overline{\bbQ}$ such that $\sigma (a) = a$ for all $a \in K$. Let $f(X) \in K[X]$ be irreducible, and let $Z_f$ be the set of its roots. The group of permutations of $Z_f$ is a finite group, which is given the discrete topology. Then $G_K$ acts on $Z_f$. We put a topology on $G_K$, so that the homomorphism of $G_K$ to the group of permutations of $Z_f$ is continuous for every such $f(X)$. Then $G_K$ is a topological group; it is compact and totally disconnected.

\begin{theorem}[Uchida, 1976]
Number fields $E$ and $F$ are isomorphic as fields if and only if $G_E$ and $G_F$ are isomorphic as topological groups.
\end{theorem}

The absolute Galois group is not well understood at all (it is considered an anabelian object). What we do understand well are abelian Galois groups. For a number field $K$ we denote by $K^{ab}$ the maximal abelian extension of $K$. This is the maximal extension which is Galois (i.e., any irreducible polynomial which has a root in $K^{ab}$ has all its roots in it), and such that the Galois group of $K^{ab}$ over $K$ is abelian. For example, the theorem of Kronecker and Weber says that $\bbQ^{ab}$ is the field generated by all the numbers $\exp(\frac{2 \pi i}{n})$, i.e., by all roots of unity. Unfortunately,

\begin{example}The abelianized Galois groups of $\bbQ(\sqrt{-2})$ and $\bbQ (\sqrt{-3})$ are isomorphic.\end{example}

\begin{theorem}[Cornelissen and Marcolli \cite{CorM}]
Let $E$ and $F$ be number fields. Then, $E$ and $F$ are isomorphic if and only if there exists an isomorphism of topological groups
\[
	\psi \colon G_E^{ab} \to G_F^{ab}
\]
such that for every character $\chi$ of $G_F^{ab}$ we have $L_{F, \chi}= L_{E , \psi \circ \chi }$, where $L_{F, \chi} $ denotes the L-function associated with $\psi$ .
\end{theorem}

Cornelissen and Marcolli make essential use of my work with Ken Davidson on multivariable dynamics \cite{DavKatMem}. At the epicenter of this interaction between number theory and non-selfadjoint operator algebras lies the concept of piecewise conjugacy and the fact that piecewise conjugacy  is an invariant for isomorphisms between certain operator algebras associated with multivariable dynamical systems.

Recall from Section 2 the context of a topological dynamical system $(X, \sigma)$ where $X$ locally compact Hausdorff space and $\sigma: X\rightarrow X$ proper continuous map or its $\ca$-algebraic analogue $(A, \alpha)$ where $A$ is a C*-algebra
and $\sigma: A\rightarrow A$ non-degenerate $*$-endomorphism.

We may consider multivariable analogues of the above concepts. A pair $(X,\sigma)$ is a multivariable dynamical system provided that
$X$ is a locally compact Hausdorff and
$\sigma = (\sigma_1 , \sigma_2 , \dots, \sigma_n)$, where
$\sigma_i : X \to X$, $1 \le i \le n$, are continuous (proper) maps.
A similar definition holds for a multivariable $\ca$-dynamical system $(A,\alpha)$.

We would like to have an operator algebra $\A$ that encodes the dynamics of $(X,\sigma)$.
Therefore $\A$ should contain a copy of $C_0(X)$ and $S_1, \dots,S_n$ satisfying
covariance relations
\[ f S_i = S_i (f \circ \sigma_i)\]
for $1 \le i \le n$ and $f \in C_0(X)$.

For $w \in \Fn$, say $w = i_k \dots i_1$, we write $S_w = S_{i_k} \dots S_{i_1}$.
The covariance algebra is
\[ \A_0 = \big\{ \sum_{w \in \Fn} S_w f_w : f_w \in C_0(X) \big\} , \]
where $\Fn$ is the free semigroup on $n$ letters.
This is an algebra since
\[ (S_v)( f S_w g) = S_{vw} (f \circ \sigma_w) g ,\]
where $\sigma_w \equiv \sigma_{i_k} \circ \dots \circ \sigma_{i_1}$.
We need a norm condition in order to complete $\A_0$. In the multivariable setting we have more than one choices. Two are the most prominent
\begin{enumerate}
\item[(1)] Isometric: $S_i^*S_i =I$  for $1 \le i \le n$
\item[(2)] Row Isometric: $ \begin{bmatrix}
 S_1 & S_2 & \dots & S_n \end{bmatrix} ^*\begin{bmatrix}
 S_1 & S_2 & \dots & S_n \end{bmatrix} =I $.

\end{enumerate}
Completing $\A_0$ using (1) yields the
semicrossed product $C_0(X) \times_\sigma \Fn$, while completing $\A_0$ using (2) yields the tensor algebra $\T_+(X,\sigma)$ (See Definition \ref{multdefn}.)


\subsection{Piecewise conjugate multisystems}

In order to classify our multivariable algebras up to isomorphism, we need a new notion of conjugacy. An obvious one would be to say that two multivariable dynamical systems $(X, \sigma)$
and $(Y, \tau)$ are \textit{conjugate} if there exists a
homeomorphism $\gamma$ of $X$ onto $Y$
and a permutation $\alpha \in S_n$ so that
$\tau_i = \gamma  \sigma_{\alpha(i)} \gamma^{-1}$ for $1 \le i \le n$. This is too strong for our purposes.

\begin{definition}[Davidson and Katsoulis \cite{DavKatMem}, 2011] We say that two multivariable dynamical systems $(X, \sigma)$ and $(Y, \tau)$ are
 \textit{piecewise conjugate} if there is a
homeomorphism $\gamma$ of $X$ onto $Y$
and an open cover $\{ \U_\alpha : \alpha \in S_n \}$ of $X$
so that for each $\alpha \in S_n$,
\[
 \gamma^{-1} \tau_i \gamma|_{\U_\alpha} = \sigma_{\alpha(i)} |_{\U_\alpha} .
\]
\end{definition}

The difference between the two concepts of conjugacy lies on the fact that
the permutations depend on the particular open set.
As we shall see, a single permutation generally will not suffice.


Here are two examples of when piecewise conjugacy implies conjugacy. Both are taken from \cite{DavKatMem}.

\begin{proposition}Let $(X, \sigma)$ and $(Y, \tau)$ be piecewise conjugate
multivariable dynamical systems.
Assume that $X$ is connected and that
 \[
 E:= \{ x \in X : \sigma_i(x)=\sigma_j(x), \mbox{for some }i\neq j  \}
 \]
 has empty interior.
Then $(X, \sigma)$ and $(Y, \tau)$ are conjugate.
\end{proposition}

For $n=2$, we can be more definitive.

\begin{proposition}. Let $X$ be connected and let $\sigma=(\sigma_1,\sigma_2)$;
and let $E$ as above.
Then piecewise conjugacy coincides with conjugacy
if and only if $\ol{X\backslash E}$ is connected.

\end{proposition}


\subsection{The multivariable classification problem}

We want to repeat the success of Theorem~\ref{DavKat1}. As it turns out we only succeed in the ``difficult" direction of that theorem: necessity of conjugacy for isomorphism. Indeed

\begin{theorem}[Davidson and Katsoulis \cite{DavKatMem}, 2011] Let $(X,\sigma)$ and $(Y,\tau)$ be two multivariable dynamical systems.
If $\T_+(X,\sigma)$ and $\T_+(Y,\tau)$ or
$\rC_0(X)\times_\sigma\Fn$ and $\rC_0(Y)\times_\tau\Fn$ are isomorphic as algebras, then the dynamical systems
$(X,\sigma)$ and $(Y,\tau)$ are piecewise conjugate.
\end{theorem}

For the tensor algebras, sufficiency holds in the following cases:

\begin{itemize}
\item[(i)] $X$ has covering dimension $0$ or $1$

\item[(ii)] $\sigma$ consists of no more than 3 maps. ($n\leq3$.)

\end{itemize}


For specificity,

\begin{theorem}[Davidson and Katsoulis \cite{DavKatMem}, 2011] Suppose that $X$ is a compact subset of $\bbR$.
Then for two multivariable dynamical systems $(X,\sigma)$
and $(Y,\tau)$, the following are equivalent:
\begin{enumerate}
\item $(X,\sigma)$ and $(Y,\tau)$ are piecewise topologically conjugate.
\item $\T_+(X,\sigma)$ and $\T_+(Y,\tau)$ are isomorphic.
\item $\T_+(X,\sigma)$ and $\T_+(Y,\tau)$ are completely isometrically isomorphic.
\end{enumerate}
\end{theorem}

The analysis of the $n=3$ case is the most demanding and
required non-trivial topological information about
the Lie group $SU(3)$.
The conjectured converse reduces to a question about the unitary
group $U(n)$.

\begin{conjecture} Let $\Pi_n$ be the $n!$-simplex with vertices indexed by $S_n$.
Then there should be a continuous function $u$ of $\Pi_n$  into $U(n)$
so that:
\begin{enumerate}
\item  each vertex is taken to the corresponding permutation matrix,
\item for every pair of partitions $(A,B)$ of the form
\[
 \{1,\dots,n\} = A_1 \dot\cup \dots \dot\cup A_m = B_1 \dot\cup \dots \dot\cup B_m,
\]
where $|A_s|=|B_s|,\ 1\le s \le m$,
let
\[ \P(A,B) = \{\alpha \in S_n : \alpha(A_s)=B_s, 1\le s \le m \} .\]
If $x = \sum_{\alpha \in \P(A,B)} x_\alpha \alpha$, then
the non-zero matrix coefficients of $u_{ij}(x)$ are
supported on $\bigcup_{s=1}^m B_s \times A_s$.
We call this the \textit{block decomposition condition}.
\end{enumerate}
\end{conjecture}

We have established this conjecture for $n=2$ and $3$ and Chris Ramsey has verified the cases $n=4,5$.

With Ken Davidson we considered only classical dynamical systems (dynamical systems over commutative $\ca$-algebras) and our notion of piecewise conjugacy applies exclusively to such systems. Motivated by the interaction between number theory and non-selfadjoint operator algebras, one wonders whether a useful analogue of piecewise conjugacy can be developed for multivariable systems over arbitrary $\ca$-algebras. The goal here is to obtain a natural notion of piecewise conjugacy that generalizes that of Davidson and Katsoulis from the commutative case while remaining an invariant for isomorphisms between non-selfadjoint operator algebras associated with such systems. This was undertaken successfully with Kakariadis in \cite{KakKatJNCG}.

 \begin{definition} Let $A$ be a unital $\ca$-algebra and let $P(A)$ be its pure state space equipped with the $w^*$-topology. The \textit{Fell spectrum} $\hat{A}$ of $A$  is the space of unitary equivalence classes of non-zero irreducible representations of $A$. (The usual unitary equivalence of representations will be denoted as $\sim$.) The GNS construction provides a surjection $P(A) \rightarrow \hat{A} $ and $\hat{A}$ is given the quotient topology.
\end{definition}

Let $A$ be a unital $\ca$-algebra $A$ and $\alpha = (a_1, \alpha_2, \dots , \alpha_{n})$ be a multivariable system consisting of unital $*$-epimorphisms. Any such system $(A, \alpha)$ induces a multivariable dynamical system $(\hat{A}, \hat{\alpha})$ over its Fell spectrum $\hat{A}$.

\begin{definition} Two multivariable systems $(A,\alpha)$ and $(B,\beta)$ are said to be \textit{piecewise conjugate on their Fell spectra} if the induced systems $(\hat{A}, \hat{\alpha})$ and $(\hat{B}, \hat{\beta})$ are piecewise conjugate, in the sense of the definition above.
\end{definition}

We have the following result with Kakariadis.

\begin{theorem}[Kakariadis and Katsoulis \cite{KakKatJNCG}, 2014] Let $(A,\alpha)$ and $(B,\beta)$ be multivariable dynamical systems consisting of $*$-epimorphisms. Assume that either $\T_{+}(A,\alpha)$ and $\T_{+}(B,\beta)$ or $A \times_{\alpha} \bbF_{n_{\alpha}}^{+}$ and $B \times_{\beta} \bbF_{n_{\beta}}^{+}$ are isometrically isomorphic. Then the multivariable systems $(A,\alpha)$ and $(B,\beta)$  are piecewise conjugate over their Fell spectra.\end{theorem}

\begin{problem} Is there an analogous result for the Jacobson spectrum?
\end{problem}

In particular this implies that when the associated operator algebras are isomorphic then both $(A,\alpha)$ and $(B,\beta)$ have the same number of \textit{$*$-epimorphisms}. (We call this property invariance of the dimension). In the commutative case, the invariance of the dimension holds for systems consisting of arbitrary endomorphisms. Is it true here?

\begin{theorem}[Kakariadis and Katsoulis \cite{KakKatJNCG}, 2014] There exist multivariable systems $(A,\alpha_1, \alpha_2)$ and $(B,\beta_1, \beta_2, \beta_3)$ consisting of $*$-monomorphisms for which  $\T_{+}(A,\alpha_1, \alpha_2)$ and $\T_{+}(B,\beta_1, \beta_2, \beta_3)$ are isometrically isomorphic.\end{theorem}

\begin{problem}[Invariance of dimension for semicrossed products]
Let $(A,\alpha)$ and $(B,\beta)$ be multivariable dynamical systems consisting of $*$-endomorphisms. Prove or disprove: if $A \times_{\alpha} \bbF_{n_{\alpha}}^{+}$ and $B \times_{\beta} \bbF_{n_{\beta}}^{+}$  are isometrically isomorphic then $n_{\alpha} =n_{\beta}$.
\end{problem}

 We say that two multivariable $\ca$-dynamical systems $(A, \alpha)$ and $(B, \beta)$ are
 \textit{outer conjugate} if they have the same dimension and there are
$*$-isomorphism $\gamma: A \rightarrow B$, unitary operators $U_i \in B$  and $\pi \in S_n$ so that
\[
 \gamma^{-1} \alpha_i \gamma (b) = U_i^*\beta_{\pi (i)}(b)U_i .
\]
for all  $b \in B$ and $i$.

\begin{theorem}[Kakariadis and Katsoulis \cite{KakKatJNCG}, 2014] \label{primitive}
Let $(A,\alpha)$ and $(B,\beta)$ be two automorphic multivariable $\ca$-dynamical systems and assume that $A$ is primitive. Then the following are equivalent:
\begin{enumerate}
\item $A\times_{\alpha} \bbF_{n_{\alpha}}^{+}$  and $B\times_{\beta} \bbF_{n_{\beta}}^{+}$ are isometrically isomorphic.
\item $\T^+(A,\alpha)$ and $\T^+(B,\beta)$ are isometrically isomorphic.
\item $(A,\alpha)$ and $(B,\beta)$ are outer conjugate.
\end{enumerate}
\end{theorem}

Let us sketch the proof that (i) or (ii) implies (iii) in the above theorem. Assume now that $(A,\alpha)$ and $(B,\beta)$ are two multivariable dynamical systems such that
$\T^+(A,\alpha)$ and  $\T^+(B,\beta)$ ( or $A\times_{\alpha} \bbF_{n_{\alpha}}^{+}$  and $B\times_{\beta} \bbF_{n_{\beta}}^{+}$) are isometrically isomorphic via a mapping $\ga$. Since $\ga$ is isometric, it follows that $\ga|_A$ is a $*$-monomorphism that maps $A$ onto $B$ (This is the only point where we use that $\ga$ is isometric.) We will be denoting $\ga|_A$ by $\ga$ as well.

Let $S_i$, $i=1,\dots  , n_{\alpha}$, (resp. $T_i$,  $i=1,2,\dots, n_{\beta}$ ) be the generators in $\T^+(A,\alpha)$ (resp. $\T^+(B,\beta)$) and let $b_{ij}$ be the $T_i$-Fourier coefficient of $\ga(s_j)$, i.e.,
\[
\ga(S_j)= b_{0j} + T_1 b_{1j} + T_2 b_{2j} + \cdots + T_n b_{nj} + Y,
\]
where $Y$ involves Fourier terms of order 2 or higher.

Since $\ga$ is a homomorphism,
\begin{align*}
\ga(a)\ga(S_j)=\ga(aS_j)=\ga(S_j \alpha_j(a))=\ga(S_j) \ga\alpha_j(a),
\end{align*}
 for all $a\in A$. Hence, $\beta_i\ga(a)b_{ij}=b_{ij}\ga\alpha_j(a)$,
$a \in A$, and so
\begin{equation*} \label{eq:conj}
\beta_i(b) b_{ij}= b_{ij} \ga\alpha_j\ga^{-1}(b)= b_{ij}\tilde{\alpha}_j(b),
\end{equation*}
for all $b \in B$.

From the intertwining equation
\begin{equation*} \label{eq:conj}
\beta_i(b) b_{ij}=  b_{ij}\tilde{\alpha}_j(b), b \in B \qquad(*)
\end{equation*}
we obtain.
\begin{itemize}
\item Since $A$ is primitive, $b_{i,j}$ is either zero or invertible!
\item If $b_{ij} \neq0$ then $\beta_i \sim \tilde{\alpha}_j$.
\end{itemize}

Therefore each equivalence class $\{ \beta_1, \beta_2, \dots , \beta_n\}$ is equivalent to exactly one class
$\{  \tilde{\alpha}_1,  \tilde{\alpha}_2, \dots ,  \tilde{\alpha}_m\}$. The proof will follow if we show that $m =n $. By way of contradiction assume that  $m<n$.

Start with an "arbitrary" n-tuple $(y_1,y_2, \dots, y_n)$. From the equation

\[
T_1y_1 +T_2 y_2+\dots+T_ny_n=  \lim_e\ga(x_e),
\]

where $x_e$ are non-commutative polynomials in $S_1, S_2, \dots, S_m$ and remembering that

\[
\ga(S_j)= b_{0j} + T_1 b_{1j} + T_2 b_{2j} + \cdots + T_n b_{nj} + Y,
\]

we obtain

\begin{align*}
y_1 & = \lim_e b_{11}x_e^1 + b_{12} x_e^2 + \cdots + b_{1m} x_e^m \\
y_2 & = \lim_e b_{21}x_e^1 + b_{22} x_e^2 + \cdots + b_{2m} x_e^m\\
& \vdots \\
y_n & = \lim_e b_{n1}x_e^1 + b_{n2} x_e^2 + \cdots + b_{nm} x_e^m.
\end{align*}

Perform Gaussian elimination to reduce this system to

\begin{align*}
\bar{y}_2& = \lim_e \bar{b}_{22} x_{e}^{2}+  \bar{b}_{23} x_{e}^{3}+\cdots + \bar{b}_{2m} x_{e}^{m}\\
\bar{y}_3& = \lim_e \bar{b}_{32} x_{e}^{2}+ \bar{b}_{33} x_{e}^{3}+ \cdots + \bar{b}_{3m} x_{e}^{m}\\
& \vdots \\
\bar{y}_n& = \lim_e \bar{b}_{n2} x_{e}^{2}+ \bar{b}_{n3} x_{e}^{3}+ \cdots + \bar{b}_{nm} x_{e}^{m}.
\end{align*}

We continue this short of ``Gaussian elimination'' and we arrive at a system that contains one column and at least two non-trivial rows of the form
\begin{align*}
w_1 & = \lim_e d_1 x_e^m \\
w_2 & = \lim_e d_2 x_e^m,
\end{align*}
where the data $(w_1, w_2)$ is arbitrary. Therefore $d_1, d_2$ are non-zero, hence invertible. By letting $w_1 =1$ we obtain that $\lim_e x_e^m = d_1^{-1}$. Therefore, if we let $w_2=0$, then we get that $0=d_2 d_1^{-1}$, which is a contradiction.

The case of tensor algebras in Theorem \ref{primitive} is a special case of the following much more general result. (See Definition \ref{multdefn} for the $\ca$-correspondences appearing below.)

\begin{theorem}[Kakariadis and Katsoulis \cite{KakKatJNCG}, 2014] Let $(A,\alpha)$ and $(B,\beta)$ be multivariable dynamical systems consisting of $*$-epimorphisms. The tensor algebras $\T_{+}(A,\alpha)$ and $\T_{+}(B,\beta)$ are isometrically isomorphic if and only if the correspondences $(X_{\alpha}, A)$ and $(X_{\beta}, B)$ are unitarily equivalent.
\end{theorem}

 In light of the above result we ask

\begin{problem} Let $(A,\alpha)$ and $(B,\beta)$ be multivariable dynamical systems consisting of $*$-monomorphisms. If the tensor algebras $\T_{+}(A,\alpha)$ and $\T_{+}(B,\beta)$ are isometrically isomorphic does it follow that the correspondences $(X_{\alpha}, A)$ and $(X_{\beta}, B)$ are unitarily equivalent?
\end{problem}

\section{Crossed products of operator algebras}
As we have seen so far, most crossed product-type constructions in the theory of non-selfadjoint operator algebras involve the action of a semigroup which rarely happens to be a group, on an operator algebra which is usually a $\ca$-algebra. There is a good reason for this and it goes back to the early work of Arveson who recognized that in order to better encode the dynamics of a homeomorphism $\sigma$ acting on a locally compact space $\X$, one should abandon group actions and instead focus on the action of $\bbZ^+$ on $C_0(\X)$ implemented by the positive iterates of $\sigma$. This theme was fully explored in Sections \ref{examples} and \ref{multivariable}.

In this section we follow a less-travelled path: we start with an arbitrary operator algebra, preferably non-selfadjoint, and we allow a whole group to act on it. It is remarkable that there have been no systematic attempts to build a comprehensive theory around such algebras even though this class includes all crossed product $\ca$-algebras.  Admittedly, our interest in group actions on non-selfadjoint operator algebras arose reluctantly as well. Indeed, apart from certain important cases, the structure of automorphisms for non-selfadjoint operator algebras is not well understood. Our initial approach stemmed from an attempt to settle two open problems regarding semi-Dirichlet algebras (which we do settle using the crossed product). We soon realized that even for very ``simple" automorphisms (gauge actions), the crossed product demonstrates an unwieldily behavior that allows for significant results.

\begin{definition} \label{admit}
Let $(\A, \G, \alpha)$ be a dynamical system and let $(\C, j)$ be a $\ca$-cover of $\A$. Then $(\C, j)$ is said to be $\alpha$-admissible, if there exists a continuous group representation $\dot{\alpha}: \G \rightarrow \Aut(\C)$ which extends the representation
  \begin{equation} \label{detailona}
  \G \ni s \mapsto j\circ \alpha_s \circ j^{-1} \in \Aut (j(\A)).
  \end{equation}
\end{definition}

Since $\dot{\alpha}$ is uniquely determined by its action on $j(\A)$, both (\ref{detailona}) and its extension $\dot{\alpha}$ will be denoted by the symbol $\alpha$.

\begin{definition}[Relative Crossed Product] \label{relative}
Let $(\A, \G, \alpha)$ be a dy-\break namical system and let $(\C, j)$ be an $\alpha$-admissible $\ca$-cover for $\A$. Then, $\A \rtimes_{\C, j, \alpha} \G$ and $\A \rtimes_{\C, j, \alpha}^r \G$ will denote the subalgebras of the crossed product $\ca$-algebras $\C \rtimes_{\alpha} \G$ and $\C \rtimes_{\alpha}^r \G$ respectively, which are generated by $C_c \big(\G , j(\A)\big) \subseteq C_c\big(\G , \C\big)$.
\end{definition}

One has to be a bit careful with Definition~\ref{relative} when dealing with an \textit{abstract} operator algebra. It is common practice in operator algebra theory to denote a $\ca$-cover by the use of set theoretic inclusion. Nevertheless a $\ca$-cover for $\A$ is not just an inclusion of the form $A \subseteq \C$ but instead a pair $(\C, j)$, where $\C$ is a $\ca$-algebra, $j : \A \rightarrow \C$ is a complete isometry and $\C = \ca(j(\A))$. Furthermore, in the case of an $\alpha$-admissible $\ca$-cover,  it seems that the structure of the relative crossed product for $\A$ should depend on the nature of the embedding $j$ and one should keep that in mind when working with that crossed product. To put it differently, assume that $(\A, \G, \alpha)$ is a dynamical system and $(\C_i, j_i)$, $i=1,2$, are $\ca$-covers for $\A$. Further assume that the representations $\G \ni s \mapsto j_i \circ \alpha_s \circ j_i^{-1} \in \Aut (j_i(\A))$ extend to $*$-representations $\alpha_i : \G \rightarrow \Aut(\C_i)$, $i=1,2$. It is not at all obvious that whenever $\C_1 \simeq \C_2$ (or even $\C_1 = \C_2$), the $\ca$- dynamical systems $(\C_i, \G, \alpha_i)$ are conjugate nor that the corresponding crossed product algebras are isomorphic. Therefore the (admittedly) heavy notation $\A \rtimes_{\C, j, \alpha} \G$ and $\A \rtimes_{\C,j, \alpha}^r \G$ seems to be unavoidable. Nevertheless, whenever there is no source of confusion, we opt for the lighter notation $\A \rtimes_{\C, \alpha} \G$ and $\A \rtimes_{\C,\alpha}^r \G$. For instance, this is the case when the $\ca$-covers involved are coming either from the $\ca$-envelope or from the universal $\ca$-algebra of $\A$, as the following result shows.

\begin{lemma} \label{delicate}
 Let $(\A, \G, \alpha)$ be a dynamical system and let $(\C_i, j_i)$ be $\ca$-covers for $\A$ with either $\C_i \simeq \cenv(\A)$, $i=1,2$, or $\C_i \simeq \cmax(\A)$, $i=1,2$. Then there exist continuous group representations $\alpha_i : \G \rightarrow \Aut(\C_i)$ which extend the representations
  \[
  \G \ni s \mapsto j_i \circ \alpha_s \circ j_i^{-1} \in \Aut (j_i(\A)), \quad i=1,2.
  \]
 Furthermore
 $\A \rtimes_{\C_1 , j_1, \alpha_1} \G \simeq  \A \rtimes_{\C_2, j_2, \alpha_2}\G$ and
$\A \rtimes_{\C_1, j_1, \alpha_1}^r \G \simeq \A \rtimes_{\C_2, j_2, \alpha_2}^r\G
$, via complete isometries that map generators to generators.
 \end{lemma}

 \begin{definition}[Full Crossed Product] \label{fulldefn}
If $(\A, \G, \alpha)$ is a dynamical system then
\[
\A \rtimes_{\alpha} \G \equiv  \A \rtimes_{\cmax(\A), \alpha} \G
\]
\end{definition}

In the case where $\A$ is a $\ca$-algebra then $\A \rtimes_{\alpha}\G$ is nothing else but the full crossed product $\ca$-algebra of $(\A, \G, \alpha)$. In the general case of an operator algebra, one might be tempted to say that $\A \cpf \simeq \A \rtimes_{\cenv(\A), \alpha} \G$. This is not so clear. First, it is not true in general that $\cmax (\A) \simeq \cenv (\A)$ and as it turns out, $\cmax(\A)$ is a much more difficult object to identify than $\cenv(\A)$. Furthermore, any covariant representation of $( \cenv(\A), \G, \alpha)$ extends some covariant representation  of $(\A, \G, \alpha)$. The problem is that the converse may not be true, i.e., a covariant representation of $(\A, \G, \alpha)$ does not necessarilly extend to a covariant representation of $( \cenv(\A), \G, \alpha)$. The identification $\A \cpf \simeq \A \rtimes_{\cenv(\A), \alpha} \G$ is a major open problem, which is resolved only in the case where $\G$ is amenable or when $\A$ is Dirichlet.

For the moment let us characterize the crossed product as the universal object for covariant representations of the dynamical system $(\A, \G, \alpha)$. In the case where $\A$ is a $\ca$-algebra, this was done by Raeburn in \cite{Raeb}. In the generality appearing below, this result is new.

\begin{theorem}[Katsoulis and Ramsey \cite{KR}, 2015] \label{full Raeb}
Let $(\A, \G , \alpha)$ be a dynamical system. Assume that $\B$ is an approximately unital operator algebra such that
\begin{itemize}
\item[(i)] there exists a completely isometric covariant representation $(j_{\A}, j_{\G})$ of $(\A, \G , \alpha)$ into $M(\B)$,
\item[(ii)] given a covariant representation $(\pi, \phist, \H)$ of $(\A, \G, \alpha)$, there is a completely contractive, non-degenerate representation $L: \B \rightarrow B(\H)$ such that $\pi = \bar{L}\circ j_{\A}$ and $\phist = \bar{L}\circ j_{\G}$, and,
\item[(iii)] $\B = \overline{\Span}\{ j_{\A}(a)\tilde{\j}_{\G}(z) \mid a \in \A, z \in C_c(\G) \},$
\end{itemize}
where
\begin{equation*}
\tilde{\j}_{\G}(z) \equiv \int_{\G} z(s) j_{\G}(s) d \mu (s), \quad \mbox{for all } z \in C_c(\G).
\end{equation*}
Then there exists a completely isometric isomorphism $\rho: B \rightarrow \A \cpf$ such that
\begin{equation} \label{twostooges}
\bar{\rho} \circ j_{\A} = i_{\A} \mbox{ and } \bar{\rho} \circ j_{\G} = i_{\G}
\end{equation}
where $(i_{\A}, i_{\G})$ is the canonical covariant representation of $(\A, \G, \alpha)$ in $M(\A \cpf)$.
\end{theorem}

In the case where $\G$ is amenable, all relative full crossed products coincide as the next result shows. Its proof requires an essential use of the theory of maximal dilations, as presented in Section 5.

\begin{theorem}[Katsoulis and Ramsey \cite{KR}, 2015] \label{r=f}
 Let $(\A, \G, \alpha)$ be a dynamical system with $\G$ amenable and let $(\C, j)$ be an $\alpha$-admissible $\ca$-cover for $\A$. Then
 \[
 \A \rtimes_{\alpha} \G  \simeq  \A \rtimes_{\C , j, \alpha}\G \simeq  \A \rtimes_{\C , j, \alpha}^r \G
 \]
 via a complete isometry that maps generators to generators.
 \end{theorem}

 One of the central problems of our theory is whether or not the identity 
  \begin{equation} \label{mainid}
 \cenv(\A \rtimes_{\alpha} \G) = \cenv(\A) \rtimes_{\alpha} \G.
 \end{equation}
 is valid. Fortunately in the case where $\G$ is an abelian group we show that the above identity is indeed valid. The case where $\G$ is discrete follows easily from the work we have done so far and from the ideas of either \cite{KakKatJFA1} in the $\bbZ$ case or more directly from \cite[Theorem 3.3]{DFK}, by choosing $P = \G$, $\tilde{\alpha} = \alpha$ and transposing the covariance relations. In the generality appearing below, the result is new and paves the way for exploring non-selfadjoint versions of Takai duality.

 \begin{theorem}[Katsoulis and Ramsey \cite{KR}, 2015] \label{abelianenv}
 Let $(\A, \G, \alpha)$ be a unital dynamical system and assume that $\G$ is an abelian locally compact group. Then
 \[
 \cenv(\A \rtimes_{\alpha} \G) \simeq \cenv(\A) \rtimes_{\alpha} \G.
 \]
 \end{theorem}
 
 Here is the promised version of Takai duality for \textit{arbitrary} operator algebras. We will make shortly an important use of that duality in our investigation for the semisimplicity of crossed products.
 
Let $(\A, \G , \alpha)$ be a dynamical system with $\G$ an abelian locally compact group. Let $\hat{\G}$ be the Pontryagin dual of $\G$. The dual action $\hat{\alpha}$ is defined on $C_c(\G, \A)$ by $\hat{\alpha}_{\gamma}(f)(s)= \overline{\gamma(s)}f(s)$, $f \in C_c(\G, \A)$, $\gamma \in \hat{\G}$. 

\begin{theorem}[Takai duality, Katsoulis and Ramsey \cite{KR}, 2015] \label{Takai duality}
 Let\break $(\A , \G, \alpha)$ be a dynamical system with $\G$ a locally compact abelian group. Then 
\[
\big( \A \cpf\big)\rtimes_{\hat{\alpha}}\hat{\G} \simeq \A \otimes \K \big( L^2 (\G)\big),
\]
where $\K \big( L^2 (\G)\big)$ denotes the compact operators on $ L^2 (\G)$ and $\A \otimes \K \big( L^2 (\G)\big)$ is the subalgebra of $\cenv(\A) \otimes \K \big( L^2 (\G)\big)$ generated by the appropriate elementary tensors.
\end{theorem}

Let us give an application of our theory to solve a problem that actually motivated our investigation. In \cite{DKDoc} Davidson and Katsoulis introduced the class of semi-Dirichlet algebras. The semi-Dirichlet property is a property satisfied by all tensor algebras and the premise of \cite{DKDoc} is that this is the actual property that allows for such a successful dilation and representation theory for the tensor algebras. Indeed in \cite{DKDoc} the authors verified that claim by recasting many of the tensor algebra results in the generality of semi-Dirichlet algebras. What was not clear in \cite{DKDoc} was whether there exist ``natural" examples of  semi-Dirichlet algebras beyond the classes of tensor and Dirichlet algebras. It turns out that the crossed product is the right tool for generating new examples of semi-Dirichlet algebras from old ones. By also gaining a good understanding on Dirichlet algebras and their crossed products we were able to answer in \cite{KR} a related question of Ken Davidson: we produced the first examples of semi-Dirichlet algebras which are neither Dirichlet algebras nor tensor algebras (Theorem \ref{neither}). Stated formally

\begin{corollary} \label{neither}
There exist semi-Dirichlet algebras which are neither \break Dirichlet nor isometrically isomorphic to the tensor algebra of any $\ca$-correspondence.
\end{corollary}

Recall the definition of the Jacobson Radical of a (not necessarily unital) ring.

\begin{definition}
Let $\R$ be a ring. The Jacobson radical $\Rad \R$ is defined as the intersection of all maximal regular right ideals of $\R$. (A right ideal $\I \subseteq \R$ is regular if there exists $e \in \R$ such that $ex - x \in \I$,  for all $x \in \R$.)
\end{definition}

An element $x$ in a ring $R$ is called right quasi-regular if there exists $y \in \R$ such that $x+y+xy=0$. It can be shown that $ x \in \Rad \R$ if and only if $xy$ is right quasi-regular for all $y \in \R$. This is the same as $1+xy$ being right invertible in $\R^1$ for all $y \in \R$. 

In the case where $\R$ is a Banach algebra we have
\[
\begin{split}
\Rad \R &= \{ x \in \R \mid \lim_n \|(xy)^n\|^{1/n} = 0, \mbox{ for all } y \in \R\} \\
		&= \{ x \in \R \mid \lim_n \|(yx)^n\|^{1/n} = 0, \mbox{ for all } y \in \R\}.
		\end{split}
		\]

A ring $\R$ is called semisimple iff $\Rad \R =\{0\}$.	

The study of the various radicals is a central topic of investigation in Abstract Algebra and Banach Algebra theory. In Operator Algebras, the Jacobson radical  and the semisimplicity of operator algebras have been under investigation since the very beginnings of the theory.

Our next result uncovers a new permanence property in the theory of crossed products.
 
\begin{theorem}[Katsoulis and Ramsey \cite{KR}, 2015]\label{firstsemisimple}
Let $(\A, \G, \alpha)$ be a dynamical system with $\G$ a discrete abelian group. If $\A$ is semisimple then $\A \rtimes_\alpha \G$ is semisimple.
\end{theorem}
\begin{proof}
Assume that the crossed product is not semisimple and so there is a nonzero $a \in \Rad \A \rtimes_\alpha \G$. Any isometric automorphism fixes the Jacobson radical and so $\Phi_g(a) = a_g \in \Rad \A \rtimes_\alpha \G$ for all $g\in \G$, where $a \sim \sum_{g\in G} a_gU_g$.
By a standard result in operator algebra theory involving the Fejer kernel,  since $a\neq 0$ there is a $g\in \G$ such that $a_g \neq 0$. This implies that $a_g b$ is quasinilpotent for all $b\in \A$ and so $a_g\in \Rad \A$. Therefore, $\A$ is not semisimple.
\end{proof}

The previous result raises two natural questions. Is the converse of Theorem \ref{firstsemisimple} true? Is Theorem~\ref{firstsemisimple} valid beyond discrete abelian groups? As we shall see shortly, both questions have a negative answer. To see this for the first question, we investigate a class of operator algebras which was quite popular in the 90's, the triangular AF algebras \cite{DavKatAdv, Don, DonH, DonHJFA, Hu, LS, Pow}.

\begin{definition}
Let $\A$ be a strongly maximal TAF algebra.
The dynamical system $(\A, \G, \alpha)$ is said to be \textit{linking} if for every matrix unit $e\in \A$ there exists a group element $g\in G$ such that $e\A \alpha_g(e) \neq \{0\}$.
\end{definition}

By Donsig's criterion if $\A$ is semisimple then $(\A, \G, \alpha)$ is linking. The following example shows that there are other linking dynamical systems.

\begin{example} \label{linking example}
Let $\A_n = \bbC \oplus \T_{2n}$ and define the embeddings $\rho_n : \A_n \rightarrow \A_{n+1}$ by
\[
\rho_n(x \oplus a) \ = \ x \oplus \left[\begin{array}{ccc} x \\ & a \\ &&x \end{array}\right].
\]
Then $\A = \varinjlim \A_n$ is a strongly maximal TAF algebra that is not semisimple. Consider the following map $\psi:\A_n \rightarrow \A_{n+1}$ given by
\[
\psi(x \oplus a) \ =  \ x \oplus \left[\begin{array}{ccc} x \\ & x \\ && a \end{array}\right].
\]
You can see that  $\psi\circ\rho_n = \rho_{n+1}\circ\psi$ on $\A_n$ and so $\psi$ is a well-defined map on $\cup \A_n$. By considering that
\[
\psi^{-1}(x \oplus a) \ = \ x \oplus \left[\begin{array}{ccc} a \\ & x \\ && x \end{array}\right]
\]
one gets $\psi\circ\psi^{-1} = \psi^{-1}\circ\psi = \rho_{n+1}\circ\rho_n$ on $\A_n$. Hence, $\psi$ extends to be an isometric automorphism of $\A$. Finally, for every $e_{i,j}^{2n} \in \A_n, i\neq j$
\[
e_{i,j}^{(2n)} \left[\begin{array}{ccc} 0_{2n}  \\ & 0_{2n} & e_{j,i}^{(2n)}\\ && 0_{2n} \end{array}\right] \psi^{(2n)}(e_{i,j}^{(2n)})
\]
\[
= \ \left[\begin{array}{ccc} 0_{2n} \\ & e_{i,j}^{(2n)}  \\ && 0_{2n} \end{array}\right]\left[\begin{array}{ccc} 0_{2n} \\ & 0_{2n} & e_{j,i}^{(2n)} \\ && 0_{2n} \end{array}\right]  \left[\begin{array}{ccc}  0_{2n} \\ & 0_{2n} \\ && e_{i,j}^{(2n)} \end{array}\right]
\]
\[
= \left[\begin{array}{ccc} 0_{2n}  \\ & 0_{2n} & e_{i,j}^{(2n)}\\ && 0_{2n} \end{array}\right] .
\]
Therefore, $(\A, \bbZ, \psi)$ is a linking dynamical system.
 \end{example}

 \vspace{.05in}

The following theorem and the previous example establish that the converse of Theorem \ref{firstsemisimple} is not true in general.

\begin{theorem}[Katsoulis and Ramsey \cite{KR}] \label{mainsemisimple}
 Let $\A$ be a  strongly maximal TAF algebra and $\G$ a discrete abelian group. The dynamical system $(\A, \G, \alpha)$ is linking if and only if $\A \rtimes_\alpha G$ is semisimple. \end{theorem}

In order to answer the other question we need the following.
 
 \begin{lemma} \label{tensorcompact}
Let $\A$ be an operator algebra and let $\K (\H)$ denote the compact operators acting on a separable Hilbert space $\H$. If $\A \otimes \K(\H)$ is semisimple, then $\A$ is semisimple.
\end{lemma}

\begin{proof}
Identify  $\A \otimes \K(\H)$ with the set of all infinite operator matrices $[(a_{i j})]_{i, j =1}^{\infty}$ with entries in $\A$, which satisfy
\[
\big\| [(a_{i j})]_{i, j =1}^{\infty} -[(a_{i j})]_{i, j =1}^{m} \big\| \xrightarrow[\phantom{x} m \rightarrow \infty \phantom{x}] {}0.
\]
By way of contradiction, assume that $0 \neq x \in \Rad \A$. Let $$X = x\otimes e_{11} \in A \otimes \K(\H)$$ be the infinite operator matrix whose $(1,1)$-entry is equal to $x$ and all other entries are $0$.

If $ A= [(a_{i j})]_{i, j =1}^{\infty} \in A \otimes \K(\H)$, then an easy calculation shows that 
\[
\begin{split}
(AX)^n &=
 \left( \begin{array}{cccc}
(a_{1 1}x)^n & 0  &0 &\dots \\
a_{2 1}x(a_{1 1}x)^{n-1}& 0 &0 & \dots\\
a_{3 1}x(a_{1 1}x)^{n-1}&0&0 & \dots\\
\vdots & \vdots &\vdots & \ddots 
\end{array} \right) \\
&= A\big( (a_{11}x)^{n-1}\otimes e_{11}\big).
\end{split}
\]
Hence 
\[
\begin{split}
\lim_n \|(AX)^n\|^{1/n}&\leq \lim_n \| A\|^{1/n} \limsup_n  \|(a_{11}x)^{n-1}\|^{1/n} \\
                                               &=\limsup_n  \|(a_{11}x)^{n}\|^{1/n} =0
                                               \end{split}
\]
because $x \in \Rad \A$. Hence $0 \neq X \in \Rad \A\otimes \K( \H)$, which is the desired contradiction.
\end{proof}

We now show that Theorem \ref{firstsemisimple} does not necessarily hold for groups which are not discrete and abelian. Using our Takai duality, we can actually show that this fails even for $\bbT$.

\begin{example} \label{Texam}
\textit{A dynamical system $(\B, \bbT , \beta)$, with $\B$ a semisimple operator algebra, for which $\B \rtimes_{\beta} \bbT$ is not semisimple.}

We will employ again our previous results and Takai duality. In Example~\ref{linking example} we saw a linking dynamical system $(\A, \bbZ, \alpha)$ for which $\A$ is not semisimple. Since $(\A, \bbZ, \alpha)$ is linking, we have by Theorem~\ref{mainsemisimple} that the algebra $\B \equiv \A\rtimes_{\alpha} \bbZ$ is semisimple. Let $\beta\equiv \hat{\alpha}$. Then,
\[
\B \rtimes_{\beta} \bbT =\big(\A \rtimes_{\alpha}\bbZ\big) \rtimes _{\hat{\alpha}}\bbT \simeq \A\rtimes \K(\ell^2(\bbZ)),
\]
which is not semismple,
\end{example}

Quite interestingly, the converse of Theorem \ref{firstsemisimple} holds for compact abelian groups. Once again, the result follows from Takai duality.

\begin{theorem}[Katsoulis and Ramsey \cite{KR}] \label{Tsemis}
Let $(\A, \G , \alpha)$ be a dynamical system, with $\G$ a compact, second countable abelian group. If $\A \rtimes_{\alpha} \G$ is semisimple, then $\A$ is semisimple.
\end{theorem}

\begin{proof} Assume that $\A \rtimes_{\alpha} \G$ is semisimple. Then Theorem~\ref{firstsemisimple} implies that $\big(\A \rtimes_{\alpha}\G \big)\rtimes_{\hat{\alpha}} \hat{\G}$ is semisimple. By Takai duality, $\A \otimes \K \big( L^2(\G)\big)$ is semisimple and so by Lemma ~\ref{tensorcompact}, $\A$ is semisimple, as desired
\end{proof}

Another natural question in the theory of crossed products asks whether or not the class of tensor algebras is being preserved under crossed products by locally compact abelian group. Our next result shows that this is not the case.

\begin{theorem}[Katsoulis and Ramsey \cite{KR}]  \label{crossedtensor}
Let $\G$ be a discrete amenable group and let $\alpha: \G \rightarrow \Aut \big( \bbA(\bbD)\big)$ be a representation. Assume that the common fixed points of the M\"obius transformations associated with $\{ \alpha_g\}_{g \in \G}$ do not form a singleton. Then $\bbA(\bbD) \rtimes_{\alpha} \G$ is a Dirichlet algebra which is not isometrically isomorphic to the tensor algebra of any $\ca$-correspondence.
\end{theorem}

Nevertheless for a special class of dynamical systems we obtain a positive answer.

\begin{theorem}[Katsoulis and Ramsey \cite{KR}] 
Let $\A$ be a tensor algebra and let $\alpha: \G \rightarrow \Aut \A$ be the action of a locally compact group $\G$ by gauge actions. Then the relative crossed product $\A \rtimes_{\cenv(\A), \alpha}\G$ is the tensor algebra of a $\ca$-correspondence.
\end{theorem}

It turns out that the above result allows us to reformulate a problem in $\ca$-algebra theory, the Hao-Ng Isomorphism Conjecture, into a problem concerning the $\ca$-envelope of a crossed product operator algebra and the validity of (\ref{mainid}). We direct the reader to \cite{KR} for more details.

\section{Local maps and representation theory}


The study of local maps, i.e., local derivation, local multipliers etc, has a long history in operator algebras  \cite{Joh, Joh2, Had, HadK, HadL2, HadL, Kad}. My involvement with such maps started somewhat accidentally. I was trying to understand whether or not a compact adjointable operator on a Hilbert $\ca$-module has a non-trivial invariant subspace. Of course it does and as it turns out the adjointable operators have lots of \textit{common} invariant subspaces provided the $\ca$-module is not a Hilbert space. Characterizing these common invariant subspaces and other peripheral results however required a result of Barry Johnson on local multipliers \cite{Joh}, which sparked my interest on the topic.

\begin{definition} If $\X$ is  Banach space and $ \S\subseteq B(\X)$, then $\S$ is said to be reflexive iff the following condition  is satisfied
\[
T \in B(\X) \mbox{,  } Tx \in [\S x], \, \forall x \in \X \implies T \in \S
\]
\end{definition}

For $\S \subseteq X$ unital algebra this is equivalent to the familiar
\[
T(M) \subseteq M, \forall M \in \Lat \S \implies T \in \S
\]

\begin{theorem}[Katsoulis \cite{Katsrefl}, 2014] Let $E$ a Hilbert C*-module over a C*-algebra $\A$ and let $L(E)$ be the adjointable operators on $E$. Then
\[
\Lat  L (E)= \{ E\J \mid \J \subseteq \overline{\sca{E,E}} \mbox{ closed left ideal } \}
\]
where
\[
E \J=\{ \xi j   \mid  \xi \in E , j  \in \J \}.
\]
and the association $\J \mapsto E\J$ establishes a complete lattice isomorphism between the closed left ideals of $\overline{\sca{E,E}} $ and $\Lat  L (E)$.\end{theorem}

Note that if $\End_{\A} (E)$ denotes the bounded $\A$-module opertors on $E$, then the above implies
\[
\Lat  L (E) = \Lat \End _{\A} (E)
\]

\begin{theorem} [Katsoulis \cite{Katsrefl}, 2014] Let $E$ be a Hilbert module over a C*-algebra $\A$. Then
\[
 \Alg \Lat \L (E) = \End _{\A} (E).
 \]
 In particular, $\End_{\A} (E)$ is a reflexive algebra of operators acting on $E$.\end{theorem}

The proof follows from the following

\begin{theorem} [Johnson \cite{Joh}, 1968]  Let $\A$ be a semisimple Banach algebra and let $\Phi$ be a linear operator acting on $\A$ that leaves invariant all closed left ideals of $\A$. Then $$\Phi (ba)=\Phi(b)a, \, \forall\, a,b \in \A,$$ i.e, $\Phi$ is a left multiplier.

In particular, if $1 \in \A$ is a unit then $\Phi$ is the left multiplication operator by $\Phi (1)$.\end{theorem}

\begin{definition} A map $S: \A \rightarrow \X$ into a right Banach $\A$-module is said to be a local left multiplier iff for every $A \in \A$, there exists left multiplier $\Phi_{a} \in LM(\A, \X)$ so that $S(a)=\Phi_{a}(a)$.
\end{definition}

Approximately local multipliers are defined to satisfy an approximate version of the above definition.

\begin{proposition} Let $\A$ be a Banach algebra with approximate unit. If $S \in B(\A)$, then the following are equivalent
 \begin{itemize}
\item[(i)] $S$ is a closed left ideal preserver
\item[(ii)] $S$ is an approximately local left multiplier
\item[(iii)] $S \in \Lat \Alg LM(\A)$.
\end{itemize}
\end{proposition}

\begin{proof}
Assume that (i) is valid and let $a \in \A$. Note that the set 
\[
\J_a \equiv \overline{ \{ ba\mid b \in \A\}\phantom{.}}^{\| .\|} \subseteq \A
\]
is a closed left ideal. Hence $S(a) \in \J_a$ and so 
\[
S(a) = \lim b_na= \lim_n L_{b_n}(a),
\]
i.e., $S$ is an approximate left multiplier.

Assume now that (ii) is true. To show that $S \in \Lat \Alg LM(\A)$, it is enough to prove that $S(\J_a) \subseteq \J_a$, for all $a \in A$. However, since $S$ is an approximately local left multiplier
\[
S(ba) = \lim_nL_{c_n}(ba) =\lim_n c_nba \in \J_a
\]
and since is bounded, we obtain (ii).

The rest of the proof follows from similar arguments.
\end{proof}

The proposition above allows us to reformulate Johnson's theorem as follows

\begin{theorem} [Johnson \cite{Joh}, 1968] The space $LM(\A)$ of left multipliers over a semisimple Banach algebra is reflexive, i.e.,$ \Alg \Lat M(\A) = \A$.\end{theorem}

\begin{question}What about Johnson's Theorem in the context of non-semisimple operator algebras? Is the space $LM(\A)$ of left multipliers over an operator algebra reflexive?
\end{question}

It is easy to produce a negative answer even for finite dimensional algebras.

\begin{example} \label{nonmult}
If
\[
\fA= \left\{
\begin{pmatrix}
\lambda & \mu \\
0 &\lambda
\end{pmatrix} \mid \lambda, \mu \in \bbC
\right\} = \left\{ \lambda I + \mu e_{12} \mid \lambda, \mu \in \bbC
\right\}
\]
then
\[
S_{\fA}:\fA \longrightarrow \fA:
\lambda I + \mu e_{12} \longmapsto
\lambda I + 2 \mu e_{12}
\]
is a local multiplier which is not a multiplier.

Indeed if $\lambda \neq 0$ then $S_{\fA}(\lambda I + \mu e_{12}) = (I + (\mu / \lambda) e_{12})(\lambda I + \mu e_{12})$ or otherwise $S_{\fA}(\mu e_{12}) = 2I (\mu e_{12}) $. This shows that $S_{\fA}$ is a local multiplier. It is easy to see that in the case $\lambda \neq 0$, the factor $(I + \mu / \lambda e_{12})$ is uniquely determined by $\lambda I + \mu e_{12}$ and so $S_{\fA}$ cannot be a multiplier.
\end{example}

\begin{proposition} \textup{(Hadwin \cite{Had, HadK} 90's)} Let $\A$ be a Banach algebra generated  by its idempotents and $\X$ be a right Banach $\A$-module. Then any approximately local left multiplier from $\A$ into $\X$ is a multiplier. Hence $LM(\A, \X)$ is reflexive.
\end{proposition}

\begin{proof} Let $S: \A \rightarrow \X$ be an approximate left multiplier. Note that for any $a, p \in \A$ with $p=p^2$,  we have $S(ap) \in \overline{\X p}$ and $S(a(1-p)) \in \overline{\X (1-p)}$ Therefore
\begin{align*}
S(a)p&= S(ap)p+S(a(1-p))p \\
	 	  &=S(ap)p= S(ap).
\end{align*}
Repeated applications of the above establish the result in the case where $p$ is a product of idempotents. Since $\A$ is generated by such, $S$ is a left multiplier, as desired.
\end{proof}

The previous result formed the basis for a variety of results by Hadwin, Li, Pan Dong and others to establish reflexivity for $LM(A)$, where $\A$ ranges over a variety of algebras rich in idempotents, including nest and CSL algebras.

 What about semicrossed products? Or tensor algebras of multivariable systems? What about their spaces of derivations? Are they reflexive? Are local derivations actually derivations? Such algebras might contain no idempotents. We use instead Representation theory

\begin{theorem}[Katsoulis \cite{Katslocal}] If $\G=  (\G^{0}, \G^{1}, r , s)$ is a topological graph, then the finite dimensional nest representations of its tensor algebra $\T_{\G}^{+}$ separate points.
\end{theorem}

The above generalizes an earlier result of Davidson and Katsoulis (2006) regarding tensor algebras of graphs.

\begin{corollary} If $\G=  (\G^{0}, \G^{1}, r , s)$ is a topological graph, then\break $LM\big(\T_{\G}^{+}\big) $  is reflexive.
\end{corollary}

\begin{proof} Let $S$ be an approximately local multiplier on $\T_{\G}^+$ and let
\[
\rho_i : \T_{\G}^+\rightarrow B(\H_i), \quad i \in \bbI
\]
separating family of representations on finite dimensional Hilbert space so that each $\rho_i(\T_{\G}^+)$ is a finite dimensional nest algebra. Furthermore, $\rho_i(\T_{\G}^+)$ is a right $\T_{\G}^+ / \ker \rho_i $-module, with the right action coming from $\rho_i$. Since $S$ preserves closed left ideals we obtain \[
S_i: \T_{\G}^+ / \ker \rho_i \longrightarrow \rho_i (\T_{\G}^+ ); a+ \ker \rho_i \longmapsto \rho_i(S(a)).
\]
However $S_i$ is an an approximate left multiplier and so Hadwin's Theorem implies that $S_i$ is a actually a left multiplier. Hence $$S_i (ab + \ker \rho_i)= S_i(a+\ker \rho_i)\rho_i(b)$$ and so
\[
\rho_i\big( S(ab)-S(a)b \big) =0, \quad \mbox{for all } i \in \bbI.
\]
Since $\cap_i \ker \rho_i= \{0\}$, the conclusion follows.
\end{proof}

\begin{corollary} A local left multiplier on $C_0(X) \times_{\sigma} \bbZ^+$ is actually a multiplier.
\end{corollary}

\begin{problem} Is the same true for $\A \times_{\sigma} \bbZ^+$ in the case where $\A$ is a non-commutative $\ca$-algebra?
\end{problem}

\begin{remark} The above Corollary is not valid for multipliers taking values in a  $C(X) \times_{\sigma} \bbZ^+$-module.
\end{remark}

\begin{theorem} [Katsoulis \cite{Katslocal}] Let $\G=  (\G^{0}, \G^{1}, r , s)$ be a  topological graph and let $\{\G_v \}_{v \in \G^0}$ be the family of discrete graphs associated with $\G$. Assume that the set of all points $ v \in \G^0$ for which $\G_v$ is either acyclic or transitive, is dense in $\G^0$. Then any approximately local derivation on $\T^+_{\G}$ is a derivation.\end{theorem}

\begin{corollary} [Katsoulis \cite{Katslocal}] Let $(X ,\sigma)$ be a dynamical system for which the eventualy periodic points have empty interior, e.g., $\sigma$ is a homeomorphism. Then any local derivation on $C_0(X) \times_{\sigma} \bbZ^+$ is actually a derivation.\end{corollary}

\begin{problem}What about the case where $\sigma$ is an arbitrary selfmap?\end{problem}



\begin{thebibliography}{99}


\bibitem{Arvenv} W. Arveson,
\textit{The non-commutative Choquet boundary}, J. Amer. Math. Soc. \textbf{21} (2008), 1065--1084.

\bibitem{Arvsub} W. Arveson,
\textit{Subalgebras of $\ca$-algebras}, Acta Math. \textbf{123} (1969), 141--224.

\bibitem{Arv1} W. Arveson,
\textit{Operator algebras and measure preserving automorphisms}, Acta Math. \textbf{118} (1967), 95--109.

\bibitem{ArvJ}  W. Arveson and K. Josephson,
\textit{Operator algebras and measure preserving automorphisms II},
J. Funct. Anal. \textbf{4} (1969), 100--134.

\bibitem{Bates}  T. Bates, D. Pask, I. Raeburn and W. Szymanski, \textit{The $\ca$-algebras of row-finite graphs}, New York J. Math. \textbf{6} (2000), 307--324.

\bibitem{BKQR}  E. Bedos, S. Kaliszewski, J. Quigg and D. Robertson,
\textit{A new look at crossed product correspondences and associated $\ca$--algebras},
J. Math. Anal. Appl. \textbf{426} (2015), 1080--1098.


\bibitem{BlLM} D. Blecher and C. Le Merdy,
\textit{Operator algebras and their modules-an operator space approach}, London Mathematical Society Monographs New Series \textbf{30}, Oxford University Press, 2004.

 
 \bibitem{CorM} G. Cornelissen and M. Marcolli, \textit{Quantum Statistical. Mechanics, L-series and Anabelian Geometry}, manuscript arXiv:1009.0736.

\bibitem{Cun} J. Cuntz, \textit{K-theory for certain $\ca$-algebras. II},  J. Operator Theory \textbf{5} (1981), 101--108.

 \bibitem{DFK} K. Davidson, A. Fuller and E.T.A. Kakariadis,
    \textit{Semicrossed products of operator algebras by semigroups}, Mem. Amer. Math. Soc., in press.


\bibitem{DavKatMem} K. Davidson and E. Katsoulis, \textit{Operator algebras for multivariable dynamics} Mem. Amer. Math. Soc. \textbf{209} (2011), no. 982, viii+53 pp.

\bibitem{DKDoc} K. Davidson and E. Katsoulis, \textit{Dilation theory, commutant lifting and semicrossed products}, Doc. Math. \textbf{16} (2011), 781--868.

\bibitem{DavKatAn} K. Davidson and E. Katsoulis, \textit{Semicrossed products of simple C*-algebras},  Math. Ann. \textbf{342} (2008), 515--525.

\bibitem{DavKatCr}  K. Davidson and E. Katsoulis, \textit{Isomorphisms between topological conjugacy algebras}, J. Reine Angew. Math.\textbf{ 621} (2008), 29--51.

\bibitem{DavKatAdv} K. Davidson and E. Katsoulis,
\textit{Primitive limit algebras and $\ca$-envelopes}, Adv. Math. \textbf{170} (2002), 181--205.

\bibitem{DavKen} K. Davidson and M. Kennedy,
\textit{The Choquet boundary of an operator system},  Duke Math. J. \textbf{164}, 2989--3004.

\bibitem{DavPit} K. Davidson and D. Pitts, \textit{The algebraic structure of non-commutative analytic Toeplitz algebras}, Math. Ann. \textbf{311} (1998), 275--303.


\bibitem{Don} A. Donsig,
\textit{Semisimple triangular AF algebras}, J. Funct. Anal. \textbf{111} (1993), 323--349.

\bibitem{DonH}  A. Donsig and A. Hopenwasser,
\textit{Analytic partial crossed products},  Houston J. Math. \textbf{31} (2005),495--527.

\bibitem{DonHJFA}  A. Donsig and A. Hopenwasser,
\textit{Order preservation in limit algebras}, J. Funct. Anal. \textbf{133} (1995), 342--394.


\bibitem{DKM}  A. Donsig, A. Katavolos  and A. Manoussos,
\textit{The Jacobson radical for analytic crossed products}, J. Funct. Anal. \textbf{187} (2001), 129--145.

\bibitem{DrMc} M. Dritschel and S. McCullough, \textit{Boundary representations for families of representations of operator algebras and spaces},  J. Operator Theory  \textbf{53}  (2005), 159--167.

\bibitem{FMR} N. Fowler, P. Muhly and I. Raeburn,
\textit{Representations of Cuntz-Pimsner algebras}, Indiana Univ. Math. J. \textbf{52} (2003) 569--605.

\bibitem{GootL} E. Gootman and A. Lazar,
\textit{Crossed products of type I AF $\ca$-algebras by abelian groups}, Isr. J. Math \textbf{56} (1986), 267--279.

\bibitem{Hamana} M. Hamana, \textit{Injective envelopes of operator systems}, Publ. Res. Inst. Math. Sci. \textbf{15} (1979), 773--785. 

\bibitem{HadH}  D. Hadwin and T. Hoover,
\textit{Operator algebras and the conjugacy of transformations.},
J. Funct. Anal. \textbf{77} (1988), 112--122.

\bibitem{HN} G. Hao and C-K. Ng,
\textit{Crossed products of $\ca$-correspondences by amenable group actions},
J. Math. Anal. Appl. \textbf{345} (2008), 702--707.

\bibitem{Hu} T. Hudson,
\textit{Ideals in triangular AF algebras}, Proc. London Math. Soc. \textbf{69} (1994), 345--376.

\bibitem{Itoh} S. Itoh,
\textit{Conditional expectations in $\ca$-crossed products}, Trans. Amer. Math. Soc, \textbf{267} (1981), 661--667.

\bibitem{Joh2}  B.E. Johnson, \textit{Local derivations on $\ca$-algebras are derivations}, Trans. Amer. Math. Soc. \textbf{353} (2001), 313--325.

\bibitem{Joh} B.E. Johnson, \textit{Centralisers and operators reduced by maximal ideals}, J. London Math.
Soc. \textbf{43} (1968), 231--233.

\bibitem{Had} The Hadwin Lunch Bunch, \textit{Local multiplications on algebras spanned by idempotents},
Linear and Multilinear Algebra \textbf{37} (1994), 259--263.

\bibitem{HadK} D. Hadwin and J. Kerr, \textit{Local multiplications on algebras},
      J. Pure Appl. Algebra  \textbf{115}  (1997), 231--239.

\bibitem{HadL2} D. Hadwin and J. Li, \textit{Local derivations and local automorphisms on some algebras}, J. Operator Theory \textbf{60} (2008), 29--44.

\bibitem{HadL} D. Hadwin and J. Li, \textit{Local derivations and local automorphisms},  J. Math. Anal. Appl. \textbf{290} (2004), 702--714.

\bibitem{Kad} R. Kadison, \textit{Local derivations}, J. Algebra \textbf{130} (1990), 494--509.

\bibitem{Kak}  E.T.A. Kakariadis,
\textit{The Dirichlet property for tensor algebras}, Bull. Lond. Math. Soc. \textbf{45} (2013), 1119--1130.

\bibitem{KakKatJFA1} E.T.A. Kakariadis and E. Katsoulis,
\textit{Semicrossed products of operator algebras and their $\ca$-envelopes},  J. Funct. Anal. \textbf{262} (2012), 3108--3124.


\bibitem{KakKatTrans} E.T.A. Kakariadis and E. Katsoulis,
\textit{Contributions to $\ca$-correspondences...}, Trans. Amer Math Soc. \textbf{364} (2012) 6605--6630.

\bibitem{KakKatJNCG} E.T.A. Kakariadis and E. Katsoulis,
\textit{Isomorphism invariants for multivariable C*-dynamics}, J. NonCommutative Geometry \textbf{8} (2014), 771--787.


\bibitem{Katslocal} E. Katsoulis, \textit{Local maps and the representation theory of operator algebras}, Trans. Amer. Math. Soc., to appear.

\bibitem{Katsrefl} E. Katsoulis, \textit{The reflexive closure of the adjointable operators}, Illinois J. Math. \textbf{58} (2014), 359--367.

\bibitem{KaKr} E. Katsoulis and D. Kribs, \textit{Isomorphisms of algebras associated with directed graphs}, Math. Ann.  \textbf{330}  (2004), 709--728.

\bibitem{KatsoulisKribsJFA} E. Katsoulis and D. Kribs, \textit{Tensor algebras of $C^*$-correspondences and their $\ca$-envelopes}, J. Funct. Anal.  \textbf{234}  (2006), 226--233.

\bibitem{KR} E. Katsoulis and C. Ramsey, \textit{Crossed products of operator algebras}, manuscript, Arxiv 1512.08162v2. 

\bibitem{KatsuraJFA}  T. Katsura,
\textit{On $\ca$-algebras associated with $\ca$-correspondences},
 J. Funct. Anal. \textbf{217} (2004), 366--401.

\bibitem{Katsura} T. Katsura, \textit{A class of $\ca$-algebras generalizing both graph algebras and homeomorphism $\ca$-algebras. I. Fundamental results}, Trans. Amer. Math. Soc.  \textbf{356}  (2004), 4287--4322.

\bibitem{LS} D. Larson and B. Solel,
\textit{Structured triangular limit algebras} Proc. London Math. Soc. \textbf{75} (1997), 177--193.

\bibitem{McM} M. McAsey and P. Muhly,
\textit{Representations of nonselfadjoint crossed products}, Proc. London Math. Soc.  \textbf{47} (1983), 128--144.

\bibitem{Mey}  R. Meyer,
\textit{Adjoining a unit to an operator algebra},  J. Operator Theory \textbf{46} (2001), 281--288.

\bibitem{M} P. Muhly,
\textit{Radicals, crossed products, and flows}, Ann. Polon. Math. \textbf{43} (1983), 35--42.

\bibitem{MS0} P. Muhly and B. Solel,
\textit{An algebraic characterization of boundary
representations},
Nonselfadjoint Operator Algebras, Operator Theory, and Related Topics,
Birkhäuser Verlag, Basel 1998, pp. 189--196.


\bibitem{MS} P. Muhly and B. Solel, \textit{Tensor algebras over $C^*$-correspondences: representations, dilations, and $C^*$-envelopes},  J. Funct. Anal.  \textbf{158}  (1998), 389--457.

\bibitem{MT} P. Muhly and M. Tomforde,
\textit{Adding tails to  $\ca$-correspondences}, Documenta Mathematica \textbf{9} (2004) 79--€"106


\bibitem{Paulsen} V. Paulsen,
\textit{Completely Bounded Maps and Operator Algebras}, Cambridge Studies in Advanced Mathematics \textbf{78}, Cambridge University Press, 2002.

\bibitem{Pet} J. Peters,
\textit{Semicrossed products of $\ca$-algebras}, J. Funct. Anal. \textbf{59} (1984), 498--534.

\bibitem{Pet2} J. Peters,
\textit{The ideal structure of certain nonselfadjoint operator algebras}, Trans. Amer. Math. Soc. \textbf{305} (1988), 333--352.

\bibitem{Pimsn}  M. Pimsner, \textit{A class of $\ca$-algebras generalizing both Cuntz-Krieger algebras and crossed products by $\bbZ$}, Free probability theory (Waterloo, ON, 1995), 189--212, Fields Inst. Commun., 12, Amer. Math. Soc., Providence, RI, 1997.

\bibitem{PW}Y. Poon and B. Wagner,
\textit{$\bbZ$-analytic TAF algebras and dynamical systems}, Houston J. Math. \textbf{19} (1993), 181--199.

\bibitem{Pop} G. Popescu, \textit{Non-commutative disc algebras and their representations}, Proc. Amer. Math. Soc.  \textbf{124}  (1996), 2137--2148.

\bibitem{Pop2} G. Popescu, \textit{Free holomorphic automorphisms of the unit ball of $B(H)^n$}, J. Reine Angew. Math. \textbf{638} (2010), 119--168.

\bibitem{Pow1} S.C. Power, \textit{Classification of analytic crossed product algebras}, Bull. London Math. Soc. \textbf{24 }(1992), 368--372.

\bibitem{Pow} S.C. Power,
\textit{Limit algebras: an introduction to subalgebras of $\ca$-algebras}, Pitman Research Notes in Mathematics Series \textbf{278}, Longman Scientific $\&$ Technical, Harlow 1992.

\bibitem{Raeb} I. Raeburn,
\textit{On crossed products and Takai duality}, Proc. Edinburgh Math. Soc. \textbf{31} (1988), 321--330.

\bibitem{Rin}  J. Ringrose,
\textit{On some algebras of operators}, Proc. London Math. Soc. \textbf{15} (1965), 61--83.

\bibitem{Will} D. Williams,
\textit{Crossed products of $\ca$-algebras}, Mathematical Surveys and Monographs, Vol. 134, American Mathematical Society, 2007.


\end{thebibliography}
\end{document}